\documentclass[a4paper, 11pt]{article}
\pdfoutput=1

\usepackage{jheppub}

\usepackage{hyperref}
\hypersetup{
  colorlinks,
  linkcolor={violet},
  citecolor={brown},
  urlcolor={gray}
}

\usepackage{graphicx}
\graphicspath{ {./} } 
\usepackage{mathrsfs}
\usepackage{amsmath}
\usepackage{enumitem}
\usepackage{mathtools}
\usepackage{amssymb}
\usepackage{stmaryrd}
\usepackage{amsthm}
\usepackage{tikz-cd}
\usepackage{tikz}
\usetikzlibrary{fit,shapes.geometric}
\usepackage{xcolor}
\usepackage{environ}
\usepackage{bm}
\usepackage{cleveref}

\input epsf.sty

\newcommand{\re}{{\rm e}}
\newcommand{\ri}{{\rm i}}

\newcommand{\e}{\mathbf{e}_1}

\def\IZ{{\mathbb Z}}
\def\IR{{\mathbb R}}
\def\IC{{\mathbb C}}
\def\IH{{\mathbb H}}
\def\IQ{{\mathbb Q}}

\def\IP{{\mathbb P}}

\newcommand{\mrho}{\mathsf{\rho}}

\newcommand{\CA}{{\cal A}}
\newcommand{\CB}{{\cal B}}
\newcommand{\CC}{{\cal C}}

\newcommand{\CS}{{\cal S}}

\def\frakg{\mathfrak{g}}
\def\frakf{\mathfrak{f}}
\def\tfrakg{\tilde{\mathfrak{g}}}
\def\tfrakf{\tilde{\mathfrak{f}}}

\newcommand{\be}{\begin{equation}}
\newcommand{\ee}{\end{equation}}
\newcommand{\ba}{\begin{aligned}}
\newcommand{\ea}{\end{aligned}}

\newcommand{\borel}{\mathcal{B}}

\newtheorem{theorem}{Theorem}[section]
\newtheorem{prop}[theorem]{Proposition}

\newtheorem{cor}[theorem]{Corollary}

\newtheorem*{definition*}{Definition}
\newtheorem{definition}{Definition}[section]
\newtheorem{conjecture}{Conjecture}
\newtheorem*{conjecture*}{Conjecture}

\newtheorem{rmk}{Remark}[section]

\theoremstyle{definition}
\newtheorem{example}{Example}

\title{\huge{\textbf Modular resurgent structures}}

\author{Veronica Fantini$^a$ and Claudia Rella$^b$}
\affiliation{${}^a$IH\'ES, 91440 Bures-sur-Yvette, France \\ ${}^b$D\'epartement de Physique Th\'eorique, Universit\'e de Gen\`eve, CH-1211 Gen\`eve, Switzerland}

\emailAdd{fantini@ihes.fr}
\emailAdd{claudia.rella@unige.ch}

\abstract{The theory of resurgence uniquely associates a factorially divergent formal power series with a collection of exponentially small non-perturbative corrections paired with a set of complex numbers known as Stokes constants.
When the Borel plane displays a single infinite tower of singularities, the secondary resurgent series are trivial, and the Stokes constants are coefficients of an $L$-function, a rich analytic number-theoretic fabric underlies the resurgent structure of the asymptotic series. We propose a new paradigm of \emph{modular resurgence} that focuses on the role of the Stokes constants and the interplay of the $q$-series acting as their generating functions with the corresponding $L$-functions.
Guided by two pivotal examples arising from topological string theory and the theory of Maass cusp forms, we introduce the notion of \emph{modular resurgent series}, which we conjecture to have specific summability properties as well as to be related to quantum modular forms.}

\makeatletter
\gdef\@fpheader{\null}
\makeatother

\begin{document}

\maketitle
\flushbottom

\section{Introduction}\label{sec:intro}
Resurgence, summability, and quantum modularity are active areas of research involving several communities, whose interactions have so far mostly been explored through examples arising from quantum Chern--Simons (CS) theory, particularly those related to the quantum invariants of knots and 3-manifolds. Motivated by the study of quantum invariants of 3-manifolds~\cite{qCS1}, quantum modular forms were introduced by Zagier in~\cite{zagier_modular}. Since then, new examples have been constructed~\cite{BKM1,BKM2,goswami2021quantum-theta}, and their role in quantum CS theory has been investigated both for 3-manifold invariants~\cite{qCS3,CCFGH,Cheng1,Cheng2,Wheeler:2023cht} and for knot invariants~\cite{Garoufalidis:2013rca,DG,Garoufalidis_Zagier_2023}. In parallel, the perturbative nature of complex CS theory has led to the study of resurgent structures and summability in quantum invariants of both 3-manifolds~\cite{qCS5,qCS7,crew-goswami-osburn} and knots~\cite{costin-garoufalidis,GGuM,GGuMW}. While the interplay between the analytic and arithmetic properties of these invariants was first described in~\cite{aritm-resurgence-G}, a general theory that integrates these three perspectives into a unified framework is still missing. In this paper, we take a first step to address this gap.

We take advantage of the abundant amount of information gathered from the in-depth resurgent analysis of the spectral trace of the toric Calabi--Yau (CY) threefold known as local $\IP^2$ performed by the second author in~\cite{Rella22}. This satisfies a full-fledged strong-weak resurgent symmetry and special modularity and summability properties that we present in the companion paper~\cite{FR1phys}. 
Generalizing the cardinal features of this example, we describe a new class of resurgent asymptotic series whose median resummation produces holomorphic quantum modular forms in the sense of Zagier~\cite{zagier_modular, zagier-talk}.
This newly proposed framework, which we refer to as \emph{modular resurgence}, deeply relies on the key role played by the Stokes constants and the $q$-series acting as their generating functions. Our results naturally lead us to present the definition of a \emph{modular resurgent series} (MRS), that is, a Gevrey-1 asymptotic series whose Borel plane displays an infinite tower of equally spaced singularities, the secondary resurgent series are trivial, and the Stokes constants are coefficients of an $L$-function. 
The resurgent structure of such series is called a \emph{modular resurgent structure}.
\begin{definition*}[Def.~\ref{def:modular_res_struct}]
A Gevrey-1 asymptotic series $\tfrakg \in \IC \llbracket y\rrbracket$ has a \emph{modular resurgent structure} if the following conditions hold.
\begin{enumerate}
    \item The Borel transform $\CB[\tfrakg] \in \IC\{\zeta\}$ has a tower of singularities at the locations $\zeta_m= \CA m$, $m\in\mathbb{Z}_{\ne 0}$, for some constant $\CA \in \IC$, in the complex $\zeta$-plane.
    \item For every $m\in\IZ_{\ne 0}$, the resurgent series at the singularity $\zeta_m$ is the constant function $S_m\in\IC$, \emph{i.e.}, the Stokes constant.
    \item The Stokes constants $S_m$, $m\in\IZ_{\ne 0}$, are the coefficients of two $L$-functions
    \be
		L_{+}(s)=\sum_{m >0} \frac{S_{m}}{m^s} \,, \quad L_{-}(s)=-\sum_{m >0} \frac{S_{- m}}{m^s} \, .
     \ee
\end{enumerate}
A Gevrey-1 asymptotic series with a modular resurgent structure is a \emph{modular resurgent series}.
\end{definition*}
It follows from the definition that, after the change of variable $y \mapsto -\tfrac{\CA}{2 \pi \ri y}$, the discontinuity of $\tfrakg(y)$ across its Stokes line equals the generating $q$-series $\frakf(y)$ of the Stokes constants with $q=\re^{2 \pi \ri y}$. The latter can be equivalently expressed as inverse Mellin transform of the $L$-functions $L_\pm(s)$.

Under mild assumptions, we show that a global network of exact relations connects the complete resurgent structures of canonical pairs of MRSs $\tfrakg, \tfrakf \in \IC[\![y]\!]$, in the same way that the weak and strong coupling perturbative expansions of the spectral trace of local $\IP^2$ are related by an exact resurgent symmetry~\cite{Rella22, FR1phys}. This construction, which we refer to as \emph{modular resurgence paradigm}, rests on the analytic properties of the $L$-functions $L_\pm(s), L_\pm'(s)$, $s \in \IC$, whose coefficients are the Stokes constants $\{S_m\}, \{R_m\}$, $m \in \IZ_{\ne 0}$, of the MRSs $\tfrakg, \tfrakf$, respectively. 
In particular, each MRS can be identified with the asymptotic expansion for $y \to 0$ of the discontinuity of the other after a change of variable, while the corresponding $L$-functions analytically continue each other as they share a combined \emph{functional equation}. We illustrate the paradigm schematically via the commutative diagram below, where $\frakf(y)$ and $\frakg(y)$ are the $q$-series with $q=\re^{2 \pi \ri y}$ that generate the Stokes constants of the asymptotic series $\tfrakg(y)$ and $\tfrakf(y)$, respectively. We refer the reader to Section~\ref{sec:resurgence--modularity} for a detailed explanation of the arguments behind it. 
\begin{equation}
\begin{tikzcd}[column sep=1.9em, row sep=2.8em]
       L_\pm'(s) \arrow[rrr,"\text{inverse Mellin}"]\arrow[ddrrrrrrr,sloped,"\text{functional equation}"] & & & \mathrm{disc} [\tfrakf](-\tfrac{\CA'}{2 \pi \ri y}) = \mathfrak{g}(y) \arrow[r,"y \rightarrow 0"] &  \tilde{\mathfrak{g}}(y) \arrow[rrr,"\text{resurgence}"] 
       & & & \{ S_m \}\arrow[dd,sloped,"\text{$L$-function}"] \\  \\
       \{ R_m \} \arrow[uu,sloped,"\text{$L$-function}"] & & &\tilde{\mathfrak{f}}(y) \arrow[lll,"\text{resurgence}"] 
       &  \mathfrak{f}(y) =  \mathrm{disc} [\tfrakg](-\tfrac{\CA}{2 \pi \ri y})  \arrow[l,"y \rightarrow 0"] & & & L_\pm(s) \arrow[lll,"\text{inverse Mellin}"]\arrow[uulllllll,sloped,swap,"\text{}"]
\end{tikzcd}
\end{equation}

Guided by a diverse set of examples that provide supporting evidence, we propose that certain $q$-series whose asymptotic expansion has a modular resurgent structure are holomorphic quantum modular forms and can be reconstructed via median resummation.
\begin{conjecture*}[Conj.~\ref{conj:quantum_modular1}]
Let $\frakg\colon\IH\to\IC$ be a $q$-series where $q=\re^{2 \pi \ri y}$. If its Mellin transform is an $L$-function and its asymptotic expansion $\tfrakg(y)$ as $y\to 0$ with $\Im(y)>0$ has a modular resurgent structure, then the \emph{median resummation} of $\tfrakg(y)$ reconstructs the original function $\frakg(y)$, that is, 
\be
\mathcal{S}_\theta^{\mathrm{med}}[\tfrakg](y)=\frakg(y) \, , \quad y\in\IH \cap \{\Re ( \re^{-\ri\theta} y)>0\} \, ,
\ee
where $\theta$ is the argument of the singularities in the Borel plane.
\end{conjecture*}
\begin{conjecture*}[Conj.~\ref{conj:quantum_modular2}]
Let $\frakg\colon\IH\to\IC$ be a $q$-series where $q=\re^{2 \pi \ri y}$. If its asymptotic expansion $\tfrakg(y)$ as $y\to 0$ with $\Im(y)>0$ has a modular resurgent structure, then the function $\frakg(y)$ is a \emph{holomorphic quantum modular form} for a subgroup $\Gamma\subseteq\mathsf{SL}_2(\mathbb{Z})$.
\end{conjecture*}
Beyond the original example of the spectral trace of local $\IP^2$ of~\cite{Rella22, FR1phys}, which provided a blueprint for this work and is revisited in Section~\ref{sec:example-P2}, our conjectural statements are proven in Theorems~\ref{thm:lewis-Zagier_QM} and~\ref{thm:lewis-Zagier_median} for a large class of examples originating from the theory of Maass cusp forms. In particular, due to Maass~\cite{Maass}, it is well-known that a Maass cusp form can be equivalently described in terms of a suitable pair of $L$-functions. Appealing to results of~\cite{lewis-Zagier--period}, we show that these pairs of $L$-functions naturally give rise to MRSs when the spectral parameter of the Maass cusp form is $1/2$.
Finally, we prove in Theorem~\ref{thm:modular} that MRSs occur in the asymptotic expansion of holomorphic functions that are built from general cusp forms at the cost of breaking modular invariance. These holomorphic functions are not in the form of $q$-series and yet are shown to verify the same summability and quantum modular properties.
We expect further examples of MRSs and evidence of our conjectures to come from the study of quantum invariants of knots and 3-manifolds, combinatorics, and topological strings on (toric) CY threefolds.

\medskip
This paper is organized as follows. In Section~\ref{sec: background}, we review the basic ingredients in the resurgent analysis of formal power series and the definition of (holomorphic) quantum modular forms. In Section~\ref{sec:resurgence--modularity}, we delve into the analytic number-theoretic characterization of those resurgent structures that give rise to quantum modular forms according to the newly introduced paradigm of modular resurgence. We then present the definition of a modular resurgent series, which we conjecture to have the quantum modularity and summability properties described above. In Section~\ref{sec:examples}, we review the example of the spectral trace of local $\IP^2$, briefly summarizing the strong-weak resurgent symmetry of~\cite{FR1phys}, prove our conjectures for a set of examples originating from the theory of Maass cusp forms, and analyse the case of more general modular forms.
Finally, in Section~\ref{sec:conclusion}, we conclude and mention further perspectives and questions to be addressed in future works.

\section{Basics of resurgence and quantum modularity} \label{sec: background}
In this section, we review the necessary background notions in resurgence and quantum modularity. 
On the one hand, the theory of resurgence was developed by \'Ecalle~\cite{EcalleI} in the early 1980s to investigate the properties of factorially divergent formal power series beyond classical asymptotics. In favorable circumstances, the resurgent analysis of asymptotic series unveils a universal mathematical structure of non-perturbative sectors involving a collection of numerical data called Stokes constants. See~\cite{diver-book} for an introduction to resurgence.
On the other hand, modular forms have a long history that can be traced back all the way to the nineteenth century and are of interest in disparate fields. On the whole, the theory of modular forms describes the invariant properties of certain $q$-series under the action of the modular group. However, there are examples of $q$-series that are not modular forms and yet whose failure of modularity is somewhat well-behaved and can be described explicitly. These belong to the class of quantum modular forms introduced by Zagier~\cite{zagier_modular}.
We point the reader to~\cite[Chapter IV]{wheeler-thesis} and references therein for an overview. 

\subsection{Resurgent asymptotic series}\label{sec:resurgence}
The Borel transform $\CB\colon z^{-\alpha}\IC[\![z]\!]\to\zeta^{-\alpha} \IC\{\zeta\}$, where $z, \zeta$ are formal variables and $\alpha \in \IR \backslash \IZ_{\geq 0}$, acts on $z$-monomials as
\begin{equation} \label{def:borel}
    \borel[z^{n-\alpha}]:=\frac{\zeta^{n-\alpha-1}}{\Gamma(n-\alpha)}\, , \quad n \in \IZ_{\ge 0} \, , 
\end{equation}
where $\Gamma(n-\alpha)$ is the gamma function. 
The action of the Borel transform above is extended by linearity to all formal power series in $z^{-\alpha}\IC[\![z]\!]$, while the Borel transform of the identity is conventionally denoted with the symbol $\delta:=\borel[1]$.
Let $\phi(z)$ be a \emph{Gevrey-1 asymptotic series} of the form 
\be \label{eq: phi}
\phi(z) = z^{-\alpha} \sum_{n=0}^{\infty} a_n z^n \in z^{-\alpha} \IC[\![z]\!] \, , \quad |a_n|\le n!|\CA|^{-n} \,,
\ee
where $\alpha \in \IR \backslash \IZ_{\ge 0}$ and $\CA \in \IC$. 
Its Borel transform, which we denote by $\hat{\phi}(\zeta):=\borel[\phi](\zeta)$ for simplicity, is the series
\be \label{eq: phihat}
\hat{\phi}(\zeta) = \sum_{k=0}^{\infty} \frac{a_k}{\Gamma(k-\alpha)} \zeta^{k-\alpha-1} \in \zeta^{- \alpha} \IC\{\zeta\} \, ,
\ee
which is convergent in the open disk centered at $\zeta=0$ of radius $|\CA|$. 
When extended to the whole complex $\zeta$-plane, known as the Borel plane, $\hat{\phi}(\zeta)$ shows a (possibly infinite) set of singularities $\zeta_{\omega} \in \IC$, which we label by the index $\omega \in \Omega$. 
A ray in the Borel plane of the form
\be \label{eq: ray}
\CC_{\theta_{\omega}} = \re^{\ri \theta_{\omega}} \IR_{\ge 0} \, , \quad \theta_{\omega} = \arg (\zeta_{\omega}) \, ,
\ee 
which starts at the origin and passes through a singularity $\zeta_{\omega}$, is called a Stokes ray.
The Borel plane is partitioned into sectors bounded by the Stokes rays in such a way that the Borel transform admits (generally different) analytic continuation in each sector.

\begin{definition}\label{def:resurgence}
A Gevrey-1 series $\phi\in z^{-\alpha}\IC[\![z]\!]$ is \emph{resurgent} if its Borel transform $\hat{\phi}\in\zeta^{-\alpha}\IC\{\zeta\}$ can be endlessly analytically continued. Namely, for every $L>0$, there is a finite set of points $\Omega_L$ on the Riemann surface of $\zeta^{-\alpha}$ such that $\hat{\phi}$ can be analytically continued along any smooth path of length at most $L$ that avoids $\Omega_L$ and starts from a fixed point in a neighborhood of the origin.
If, additionally, the Borel transform of $\phi$ has only simple poles and logarithmic branch points, it is called \emph{simple resurgent}.
\end{definition}

Let us assume that the formal power series in Eq.~\eqref{eq: phi} is simple resurgent. If the singularity $\zeta_{\omega}$ is a simple pole, the local expansion around $\zeta=\zeta_{\omega}$ of the Borel transform in Eq.~\eqref{eq: phihat} has the form\footnote{By slight abuse of notation, we also denote as $\hat{\phi}$ the analytic continuation of the Borel transform.} 
\be \label{eq: Stokes0}
\hat{\phi}(\zeta) = - \frac{S_{\omega}}{2 \pi \ri (\zeta - \zeta_{\omega})} + \text{regular in $\zeta-\zeta_\omega$} \, ,
\ee
where $S_{\omega} \in \IC$ is the Stokes constant at $\zeta_{\omega}$.
Whereas, if the singularity $\zeta_{\omega}$ is a logarithmic branch point, the local expansion around $\zeta=\zeta_{\omega}$ of the Borel transform is
\be \label{eq: Stokes-log}
\hat{\phi}(\zeta) = - \frac{S_{\omega}}{2 \pi \ri} \log(\zeta - \zeta_{\omega}) \hat{\phi}_{\omega}(\zeta - \zeta_{\omega}) + \text{regular in $\zeta-\zeta_\omega$} \, ,
\ee
where again $S_{\omega} \in \IC$ is the Stokes constant and the formal power series
\be \label{eq: phihat2}
\hat{\phi}_{\omega}(\zeta) = \sum_{k=0}^{\infty} \hat{a}_{k, \omega} \zeta^{k-\beta} \in \zeta^{-\beta}\IC\{\zeta \}\, ,
\ee
where $\beta\in\IR\setminus\IZ_{\ge 0}$, can be regarded as the Borel transform of the Gevrey-1 asymptotic series 
\be \label{eq: phi2}
\phi_{\omega}(z) = z^{-\beta} \sum_{n=0}^{\infty} a_{n, \omega} z^n \in z^{-\beta} \IC[\![z]\!] , \, \quad a_{n, \omega} = \Gamma(n-\beta+1) \, \hat{a}_{n, \omega} \, .
\ee
Note that the value of the Stokes constant $S_{\omega}$ depends on the normalization of $\phi_{\omega}$.

\begin{rmk}\label{rmk:new_series}
One of the fundamental characteristics of resurgent asymptotic series is the emergence of new series from the behavior of their analytically continued Borel transforms near the singularities. In particular, in the case of a simple pole, the Stokes constant, namely, the residue $S_\omega$ in Eq.~\eqref{eq: Stokes0}, contains all the resurgent information. When the singularity is, instead, a logarithmic branch point, a new germ of an analytic function, namely, the series $S_\omega \hat{\phi}_\omega$ in Eq.~\eqref{eq: Stokes-log}, emerges from the local behavior of the Borel transform.
\end{rmk}

If the analytic continuation of the Borel transform in Eq.~\eqref{eq: phihat} does not grow too fast at infinity at an angle $0 \le \theta < 2 \pi$ in the Borel plane,\footnote{We require that $\hat{\phi}(\zeta)$ grows at most exponentially in a tubular neighbourhood of the Borel plane containing the angle $\theta$~\cite[Section~5.6]{diver-book}.} its Laplace transform along the direction $\theta$ gives the Borel--Laplace sum of the original, divergent formal power series $\phi(z)$ in the same direction, which we denote by $s_{\theta}[\phi](z)$.
Explicitly,
\be \label{eq: Laplace}
s_{\theta}[\phi](z) = \int_0^{\re^{\ri \theta} \infty} \re^{-\zeta/z} \hat{\phi}(\zeta) \, d \zeta = z \int_0^{\re^{\ri \theta} \infty} \re^{-\zeta} \hat{\phi}(\zeta z) \, d \zeta \, , 
\ee
whose asymptotics as $z\to 0$
reproduces $\phi(z)$~\cite[Theorem~5.20]{diver-book}. 
\begin{definition}
Let $\phi\in z^{-\alpha}\IC[\![z]\!]$ be a resurgent Gevrey-1 asymptotic series. If the Borel--Laplace sum of $\phi$ at an angle $0 \le \theta < 2 \pi$, defined as in Eq.~\eqref{eq: Laplace}, exists in some region of the complex $z$-plane, we say that the series $\phi$ is Borel--Laplace summable along the direction $\theta$. 
\end{definition}
Note that the Borel--Laplace sum inherits the sectorial behavior of the Borel transform. Specifically, it is a locally analytic function with discontinuities across the special rays identified by 
\be
\arg(z)=\arg(\zeta_{\omega}) \, , \quad \omega \in \Omega \, .
\ee 
\begin{definition}\label{def:disc}
The discontinuity across an arbitrary ray $\CC_{\theta} = \re^{\ri \theta} \IR_{\ge 0}$ in the complex $z$-plane is the difference between the Borel--Laplace sums along two rays that lie slightly above and slightly below $\CC_{\theta}$. Namely,
\be \label{eq: disc}
\mathrm{disc}_{\theta}[\phi](z) = s_{\theta_+}[\phi](z) - s_{\theta_-}[\phi](z) = \int_{\mathcal{C}_{\theta_+} - \, \mathcal{C}_{\theta_-}} \re^{-\zeta/z} \hat{\phi}(\zeta) \,  d \zeta \, , 
\ee
where $\theta_{\pm}= \theta \pm \epsilon$ for some small positive angle $\epsilon$.
\end{definition}
A standard contour deformation argument implies that the two lateral Borel--Laplace sums differ by exponentially small terms. More precisely, if $\hat{\phi}(\zeta)$ has only simple poles, then
\be \label{eq: Stokes1-poles}
\mathrm{disc}_{\theta}[\phi](z) = \sum_{\omega  \in \Omega_{\theta}} S_{\omega} \re^{-\zeta_{\omega}/z}  \, ,
\ee
where the index $\omega \in \Omega_{\theta}$ labels the singularities $\zeta_{\omega}$ such that $\arg (\zeta_{\omega}) = \theta$ and the complex numbers $S_{\omega}$ are the same Stokes constants that appear in Eq.~\eqref{eq: Stokes0}.
Similarly, if there are only logarithmic branch points, then
\be \label{eq: Stokes1}
\mathrm{disc}_{\theta}[\phi](z) = \sum_{\omega  \in \Omega_{\theta}} S_{\omega} \re^{-\zeta_{\omega}/z} s_{\theta_-}[\phi_{\omega}](z) \, ,
\ee
where again the complex numbers $S_{\omega}$ are the Stokes constants appearing in Eq.~\eqref{eq: Stokes-log}, while $\phi_{\omega}(z)$ is the formal power series in Eq.~\eqref{eq: phi2}.

\subsubsection*{Effectiveness of the median resummation}
In the study of the summability properties of factorially divergent formal power series that arise as asymptotic limits of well-defined holomorphic functions, it is natural to ask which summability method is efficient in reproducing the original function from its asymptotic expansion. 
In some examples, such as the ones discussed in~\cite{borel_reg}, the Borel--Laplace sum is effective. However, that is not always the case. In order to reconstruct certain $q$-series from their asymptotic expansion, one has to consider the so-called median resummation. 
\begin{definition}\label{def:median}
The median resummation of the Gevrey-1 asymptotic series $\phi\in\IC\llbracket z\rrbracket$ in Eq.~\eqref{eq: phi} across an arbitrary ray $\CC_{\theta} = \re^{\ri \theta} \IR_{\ge 0}$ in the complex $z$-plane is defined as the average of the two lateral Borel--Laplace sums $s_{\theta_\pm}[\phi](z)$, that is,
\be \label{eq: median}
\mathcal{S}^{\mathrm{med}}_{\theta}[\phi](z) := \frac{s_{\theta_+}[\phi](z) + s_{\theta_-}[\phi](z)}{2} \,.
\ee
\end{definition}
Notice that $\mathcal{S}^{\mathrm{med}}_{\theta}[\phi](z)$ is an analytic function for $\arg z\in (\theta-\tfrac{\pi}{2},\theta+\tfrac{\pi}{2})$. 
In addition, by a contour deformation argument, we can equivalently write it as
\begin{equation} \label{eq: median2}
    \mathcal{S}^{\mathrm{med}}_{\theta}[\phi](z) =\begin{cases}
        s_{\theta_-}[\phi](z)+\frac{1}{2}\,\mathrm{disc}_{\theta}[\phi](z) \, , \quad & \Re \left( \re^{-\ri\theta_{-}}z \right)>0 \, , \\
        & \\
        s_{\theta_+}[\phi](z)-\frac{1}{2}\,\mathrm{disc}_{\theta}[\phi](z) \, , \quad & \Re \left( \re^{-\ri\theta_{+}}z \right)>0 \, ,
    \end{cases}
\end{equation}
where $\Re \left( \re^{-\ri\theta} z \right)>0$.
The expression in Eq.~\eqref{eq: median2} highlights how the discontinuities can be interpreted as corrections to the standard Borel--Laplace sums $s_{\theta_\pm}[\phi](z)$. In addition, by varying the contour of integration of the Laplace transform, the domain of analyticity of the median resummation can be extended. 

\begin{rmk}
    In the context of quantum CS theory, Costin and Garoufalidis conjectured that quantum invariants, such as the Kashaev invariant for knots and the WRT invariant for $3$-manifolds, are reconstructed by the median resummation of their Gevrey-1 asymptotic series at roots of unity~\cite[Conjecture~1]{costin-garoufalidis}. This was proven in~\cite[Theorem~3]{qCS7} for Seifert fibered homology spheres and more recently in~\cite[Corollary~1.5]{crew-goswami-osburn} for certain torus knots. As we show in~\cite[Theorem~4.6]{FR1phys}, further evidence of the effectiveness of the median resummation comes from the study of the spectral traces of quantum-mechanical operators obtained by quantizing the mirror curve to a toric CY threefold.
\end{rmk}

\subsection{Quantum modular forms} \label{sec:modularity}
A first description of quantum modular forms was given by Zagier in~\cite{zagier_modular}. Namely, a function $f\colon\IQ\to\IC$ is called a quantum modular form of weight $\omega\in\frac{1}{2}\IZ$ with respect to a subgroup $\Gamma\subseteq\mathsf{SL}_2(\IZ)$ if the cocycle
\begin{equation}\label{cocycle}
     h_\gamma[f](y):=(cy+d)^{-\omega} f\left(\frac{ay+b}{cy+d}\right) - f(y)
\end{equation} 
is \emph{better behaved} than $f(y)$ for all choices of $\gamma=\left( \begin{smallmatrix}
    a & b\\
    c & d
\end{smallmatrix} \right)\in\Gamma$. More precisely, $h_\gamma[f](y)$ has better analytic properties than the function $f(y)$ itself---\emph{e.g.}, being real analytic over $\IR\setminus\{\gamma^{-1}(\infty)\}$. Quantum modular forms of weight zero are called quantum modular functions.

\begin{example} 
Some examples of quantum modular forms are built from (classical) modular forms. For instance, it was proven in~\cite{zagier2001vassiliev} that the function $g\colon\IQ\to\IC$ defined by 
\be\label{eq:trefoil}
g(y):=q^{1/24}\Big(1+\sum_{n=1}^\infty (1-q)(1-q^2)\ldots(1-q^n) \Big)\,, \quad q=\re^{2\pi\ri y}\,,
\ee
which is closely related to the Kashaev invariant for the trefoil knot~\cite{costin-garoufalidis}, is a weight $3/2$ quantum modular form for $\mathsf{SL}_2(\IZ)$. Its cocycle is written explicitly as
\be\label{eq:cocycle-eta}
h_\gamma[g](y)=\int_{\gamma^{-1}(\infty)}^\infty \eta(t)(y-t)^{-3/2} dt\, , \quad \gamma\in\mathsf{SL}_2(\IZ) \, , 
\ee
where $\eta(y)=q^{1/24}\prod_{n=1}^\infty (1-q^n)=\sum_{k=1}^\infty\chi(k) q^{k^2/24}$ is the Dedekind eta function and $\chi$ is the Dirichlet character modulo $12$ with $\chi(\pm 1)=1$ and $\chi(\pm 5)=-1$.
As shown in~\cite{zagier2001vassiliev}, the function $g$ agrees to all orders in $q$ with the radial limit of the Eichler integral of $\eta$, that is, 
\be\label{eq:strange_id}
g(y)``=" -\frac{1}{2}\sum_{k=1}^\infty k\chi(k) q^{k^2/24} \, .
\ee
Besides, the cocycle in Eq.~\eqref{eq:cocycle-eta} is real analytic on $\IR\setminus\{\gamma^{-1}(\infty)\}$ and is expressed in terms of the original weight-1/2 modular form $\eta$. This is a common feature of those quantum modular forms that arise from a classical counterpart and is manifest in several examples. See~\cite[Examples~1,~3,~4]{zagier_modular} and~\cite{bringmann-q-hyper,goswami2021quantum-theta,BKM1,BKM2}. Other instances of quantum modular forms arising from the quantum invariants of hyperbolic knots have cocycles with a more complicated structure~\cite{garoufalidis_zagier_2023knots}, and whether they are related to modular forms is still unknown.
\end{example}

The Eichler integral appearing on the right-hand side (RHS) of Eq.~\eqref{eq:strange_id} is well-defined and holomorphic in the upper half of the complex plane, which we denote by $\IH = \IH_+$. In fact, it is natural to ask whether holomorphic functions on $\IH$ can be thought of as quantum modular forms in an appropriate way. Although a holomorphic function $f\colon\IH\to\IC$ already has good analyticity properties, a suitable notion of quantum modularity is obtained by requiring that the cocycle $h_\gamma[f]$ in Eq.~\eqref{cocycle} is analytic in a domain larger than $\IH$. 
Quantum modular forms that belong to this category are called holomorphic quantum modular forms, as defined by Zagier~\cite[min 21:45]{zagier-talk}.
\begin{definition} \label{def: holoQM}
    A holomorphic function $f\colon\IH\to\IC$ is a \emph{holomorphic quantum modular form} of weight $\omega\in\frac{1}{2}\IZ$ for a subgroup $\Gamma\subseteq\mathsf{SL}_2(\IZ)$ if the cocycle $h_\gamma[f]\colon\IH\to\IC$ in Eq.~\eqref{cocycle} extends holomorphically to\footnote{Note that $\IC_\gamma = \IC \setminus \left(-\infty; \, -d/c\right]$ when $c>0$ and $\IC_\gamma = \IC \setminus \left[-d/c; \, +\infty \right)$ when $c<0$.}  
    \be
    \IC_\gamma:=\{y\in\IC\colon cy+d\in\IC'=\IC\setminus\IR_{\leq 0}\}
    \ee
    for every $\gamma=\left( \begin{smallmatrix}
    a & b\\
    c & d
\end{smallmatrix} \right) \in\Gamma$.
\end{definition}
We remark that the definition above can be analogously written for a function $f\colon\IH_{-}\to\IC$, where $\IH_{-}$ denotes the lower half of the complex plane. 

\begin{example}
Following~\cite{zagier_modular}, let $\Delta\colon\IH\to\IC$ be the modular discriminant of the Dedekind eta function, that is, 
\be
\Delta(y):=\eta(y)^{24}= q\prod_{n=1}^\infty (1-q^n)^{24}\, ,\quad q=\re^{2\pi\ri y} \, .
\ee
It was shown in~\cite[Example~2]{zagier_modular} that its Eichler integral $\tilde{\Delta}\colon\IH\to\IC$, which is given by
\be\label{eq:delta-tilde}
\tilde{\Delta}(y) = -\frac{(2\pi\ri)^{11}}{10!}\int_{y}^\infty \Delta(t) (y-t)^{10} dt \,,
\ee
is a holomorphic quantum modular form for $\mathsf{SL}_2(\IZ)$ of weight $-10$. Its cocycle 
\be\label{eq:cocycle-Delta}
h_\gamma[\tilde{\Delta}](y)=\frac{(2\pi\ri)^{11}}{10!}\int_{\gamma^{-1}(\infty)}^\infty \Delta(t) (y-t)^{10} dt \,, \quad \gamma\in\mathsf{SL}_2(\IZ) \, ,
\ee
is known to be a polynomial of degree less than or equal to $10$ and therefore holomorphic on $\IC$. This is an exceptional example since, typically, the cocycle only extends to some cut plane, \emph{e.g.}, $\IC\setminus\IR_{\leq 0}$.
\end{example}

Recall that the adjective \emph{quantum} attached to modular forms indicates the failure of invariance under the action of a subgroup $\Gamma\subseteq\mathsf{SL}_2(\IZ)$. 
Quantum modular forms therefore constitute a much larger set than classical modular forms, and fully characterizing them is a rather complicated endeavor. However, when the failure of modularity is measured by a better-behaved analytic function with divergent asymptotics, the theory of resurgence appears to be a remarkably useful tool. In this paper, we put forward a proposal to identify a common structure behind the divergent asymptotic behavior of the cocycle via resurgence and exploit it to gain further understanding and control over these quantum modular objects. This argument is detailed in Section~\ref{sec: modular-structures} and evidence from the spectral theory of CY threefolds and the theory of Maass cusp forms and their periods is discussed in Sections~\ref{sec:example-P2} and~\ref{sec:example-Maass}. 

\section{Modular resurgence}\label{sec:resurgence--modularity}
In this section, we consider a particular class of Gevrey-1 asymptotic series that replicate the central features of the resurgence of the spectral trace of local $\IP^2$, which has been extensively studied in~\cite{Rella22, FR1phys}. Backed by this and other examples discussed in Section~\ref{sec:examples}, we propose a new perspective on the resurgent structure of these asymptotic series by taking the point of view of the $q$-series acting as generating functions of the Stokes constants and the $L$-functions acting as their Dirichlet series. As we shall see, a cardinal role in this paradigm of \emph{modular resurgence} is played by the functional equation satisfied by these \emph{resurgent $L$-functions}. When all Borel singularities are simple poles, we describe a global, exact resurgent symmetry connecting pairs of what we define as \emph{modular resurgent series} and present a conjectural relation between MRSs and quantum modular forms. 

\subsection{New perspectives on resurgence via \texorpdfstring{$L$}{L}-functions}\label{sec:new-resur-Lfunct}
One of the main features of resurgent asymptotic series is that one can compute the non-perturbative information hidden in the divergence of their coefficients by studying the structure of singularities in the Borel plane. Namely, starting from a germ at the origin, that is, the Borel transform of the asymptotic series, and analytically continuing it to the whole complex plane, it is possible to reconstruct a new germ at each singular point, as briefly reviewed in Section~\ref{sec:resurgence}. This collection of resurgent germs of analytic functions dictates the exponentially small corrections to be added to the original, divergent formal power series.

In the example of the spectral trace of local $\IP^2$, the resurgent structure appears trivial at first glance, as the germs at the singularities are all constants. Yet, as first discovered in~\cite{Rella22} and fully decoded in~\cite{FR1phys}, when globally viewed as a sequence, these constants possess rich analytic number-theoretic properties that make contact with the realm of $L$-functions and quantum modular forms. Furthermore, when considering both weak and strong coupling limits, a unique strong-weak resurgent symmetry holds, exchanging the perturbative and non-perturbative data in the two regimes. Since these features can be generalized to a whole new class of asymptotic series, the results of~\cite{Rella22, FR1phys} encourage us to adopt a novel perspective and motivate the following discussion.

Let us start by clarifying the notions of $L$-series and $L$-function adopted throughout this work. 
\begin{definition}
A \emph{Dirichlet series} is a formal series of the type 
\be \label{eq:dirichlet-def}
F(s)=\sum_{m=1}^\infty \frac{A_m}{m^s} \, , 
\ee
where $s$ is a formal variable and $\{A_m\}$ is a sequence of complex numbers. 
If the series converges absolutely in the right half-plane $\{\Re(s)>\alpha\} \subset \IC$, where
\be \label{eq:dirichlet-conv}
\alpha:= {\rm inf}\left\{ \alpha \in \IR \, : \, F(s) \mbox{ converges absolutely for } \Re(s)>\alpha \right\} \, ,
\ee
then $F(s)$ defines a holomorphic function there and $\alpha$ is called \emph{abscissa of (absolute) convergence}. If the coefficients are multiplicative, \emph{i.e.}, $A_{mn}=A_m A_n$ whenever $m$, $n$ are coprime, then $F(s)$ satisfies the \emph{Euler product expansion}\footnote{If the coefficients are completely multiplicative, \emph{i.e.}, $A_{mn}=A_m A_n$ for all $m$, $n$, one has $F(s)= \prod_{p \in \IP} (1-A_p/p^s)^{-1}$ for $\Re(s)>\alpha$.}
\be \label{eq:euler}
F(s)= \prod_{p \in \IP} \sum_{e=0}^\infty\frac{A_{p^e}}{p^{e s}} \, , \quad \Re(s)>\alpha \, ,
\ee
where $\IP$ is the set of prime numbers.
A Dirichlet series that converges absolutely in a right half-plane and satisfies the Euler product in that region is called \emph{$L$-series}. 
When it exists, the meromorphic continuation of an $L$-series to a larger domain containing $\{\Re(s)<0\} \subset \IC$ defines an \emph{$L$-function}.
\end{definition}
If the Dirichlet series in Eq.~\eqref{eq:dirichlet-def} converges absolutely with abscissa $\alpha$, some moderate bounds on the growth of the coefficients and their partial sums as $m \to \infty$ follow. In particular, 
\be
 \forall\varepsilon>0 \, , \quad |A_m| = o (m^{\alpha+\varepsilon}) \, , \quad \sum_{m=1}^N |A_m| = O (N^{\alpha+\varepsilon}) \, .
\ee

We can now formally present and prove the relevant statements that serve as building blocks in the new paradigm of \emph{modular resurgence}.
The following result is a generalization of the exact large-order relations of~\cite[Sections~4.2.3 and~4.3.3]{Rella22}.
\begin{prop}\label{prop:stokes}
Let $\tfrakg(y)=\sum_{n=1}^\infty a_n y^n \in \IC [\![y]\!]$
be a Gevrey-1 asymptotic series such that
\be \label{eq: cn-Am}
a_n=\frac{\Gamma(n)}{2 \pi \ri} \sum_{m \in \IZ_{\ne 0}} \frac{S_m}{\zeta_m^n} \, , \quad n \in \IZ_{> 0} \, ,
\ee
where $\zeta_m$, $S_m$, $m \in \IZ_{\ne 0}$, are complex numbers satisfying $\sum_{m \in \IZ_{\ne 0}} |S_m\, \zeta_m^{-n}| < \infty$ for all $n\in \IZ_{>0}$.
Then, the Borel transform $\CB[\tfrakg] \in \IC\{\zeta\}$ is a resurgent function with only simple poles at the locations $\zeta_m$ and corresponding Stokes constants $S_m$. 
\end{prop}
\begin{proof}
By definition, the Borel transform of $\tfrakg$ is the series
\be \label{eq: borelftilde}
\CB [\tfrakg](\zeta)=\sum_{n=1}^\infty \frac{a_n}{\Gamma(n)}\zeta^{n-1}= \frac{1}{2 \pi \ri}\sum_{n=1}^\infty \sum_{m \in \IZ_{\ne 0}} \frac{S_m}{\zeta_m^n} \zeta^{n-1}  \in \IC\{\zeta \}  \, .
\ee
Since $\tfrakg$ is Gevrey-1, the coefficients $a_n$ have factorially divergent growth, \emph{i.e.},
\be
|a_n|\leq \Gamma(n) |\CA|^{-n+1} \, , \quad n \ge 1 \, ,
\ee
for some constant $\CA\in\IC$. Hence, the $n$-series in Eq.~\eqref{eq: borelftilde} converges absolutely for $|\zeta| < |\CA|$. Moreover, for a fixed value of $n$, the $m$-series is absolutely convergent by assumption.
It follows that we can exchange the order of summation and obtain
\be\label{eq: borelftilde2}
\CB [\tfrakg](\zeta)=\frac{1}{2 \pi \ri}\sum_{m \in \IZ_{\ne 0}} \frac{S_m}{\zeta_m} \sum_{n=1}^\infty \left(\frac{\zeta}{\zeta_m}\right)^{n-1}
=- \frac{1}{2 \pi \ri}\sum_{m \in \IZ_{\ne 0}} \frac{S_m}{\zeta-\zeta_m} \, , 
\ee
which is a meromorphic function with simple poles at $\zeta=\zeta_m$ and Stokes constants $S_m$. 
\end{proof}
Note that Eq.~\eqref{eq: cn-Am} can be equivalently written as
\be \label{eq: cn-Am2}
a_n=\frac{\Gamma(n)}{2 \pi \ri \, \CA^n} \sum_{m \in \IZ_{\ne 0}} \frac{S_m}{m^n} \, , \quad n \in \IZ_{>0} \, ,
\ee
when $\zeta_m = \CA m$, $m \in \IZ_{\ne 0}$, for some constant $\CA \in \IC$.  
Besides, the converse of Proposition~\ref{prop:stokes} holds.

\begin{prop} \label{prop:stokes2}
Let $\tfrakg(y)=\sum_{n=1}^\infty a_n y^n\in\IC[\![y]\!]$ be a Gevrey-1 asymptotic series such that the Borel transform $\CB[\tfrakg] \in \IC\{\zeta\}$ is a resurgent function with only simple poles at the locations $\zeta_m$ and corresponding Stokes constants $S_m$, where $\zeta_m$, $S_m$, $m \in \IZ_{\ne 0}$, are complex numbers satisfying $\sum_{m \in \IZ_{\ne 0}} |S_m\, \zeta_m^{-n}| < \infty$ for all $n\in \IZ_{>0}$.
Then, the coefficients $a_n$ are given by Eq.~\eqref{eq: cn-Am}.
\end{prop}
\begin{proof}
Since the Borel transform of $\tfrakg$ has only simple poles at $\zeta=\zeta_m$ with residues $-S_m/(2 \pi \ri)$ for $m \in \IZ_{\ne 0}$, and no additional holomorphic part, we can write it as
\be
\CB[\tfrakg](\zeta)=-\frac{1}{2\pi\ri}\sum_{m\in\IZ_{\neq 0}}\frac{S_m}{\zeta-\zeta_m} 
\, ,
\ee
which is absolutely convergent away from the poles by assumption.
For $|\zeta|<\mbox{inf}_m|\zeta_m|$, we can expand the geometric series on the RHS around $\zeta=0$ and exchange the summation order, yielding the $n$-series in Eq.~\eqref{eq: borelftilde}. The expected formula for $a_n$ then follows. 
\end{proof}

The exact coefficient expression in Eq.~\eqref{eq: cn-Am2} can be written in the form
\be \label{eq: cn-Am3}
a_n=\frac{\Gamma(n)}{2 \pi \ri \, \CA^n} \left( L_+(n) - (-1)^n L_-(n) \right) \, , \quad n \in \IZ_{>0} \, ,
\ee
where we have introduced the Dirichlet series
\be \label{eq:lseries-pm}
L_{+}(s)=\sum_{m >0} \frac{S_{m}}{m^s} \,, \quad L_{-}(s)=-\sum_{m >0} \frac{S_{- m}}{m^s}
\ee
for $s \in \IC$.
Thus, for each positive integer $n$, the coefficient $a_n$ is obtained by evaluating the Dirichlet series of the Stokes constants at $s=n$, up to simple prefactors. 
Recall that Eq.~\eqref{eq: cn-Am3} requires the convergence condition included in Propositions~\ref{prop:stokes} and~\ref{prop:stokes2}, that is, $\sum_{m \in \IZ_{\ne 0}} |S_m\, m^{-n}| < \infty$, which holds, for instance, when the Dirichlet series in Eq.~\eqref{eq:lseries-pm} have abscissa $\alpha < 1$. If $\alpha=1$, then either $\Re(s)=\alpha$ is included in the domain of absolute convergence of $L_\pm(s)$, or the formula in Eq.~\eqref{eq: cn-Am3} diverges for $n=1$ and therefore only applies to $n\ge 2$. More generally, the exact large-order relation in Eq.~\eqref{eq: cn-Am3} gives a finite answer for $a_n$ when $n>\alpha$ and diverges for $n<\alpha$.
Whenever the Dirichlet series above have a right half-plane of absolute convergence, an additional result can be proven.
\begin{prop}\label{prop:inverse Mellin}
Let $\{S_m\}$, $m \in \IZ_{\ne 0}$, be a sequence of complex numbers such that the Dirichlet series $L_\pm (s)$ in Eq.~\eqref{eq:lseries-pm} converge absolutely in $\{\Re (s)>\alpha \}\subset\IC$.
Let $q=\re^{2 \pi \ri y}$ and 
\begin{equation}\label{eq:f_Am}
\frakf(y)=\begin{cases}
\displaystyle\sum_{m>0}S_m q^m & \mbox{if} \quad \Im(y)>0  \\
& \\
-\displaystyle\sum_{m<0}S_m q^m & \mbox{if} \quad \Im(y)<0 
\end{cases}
\end{equation} 
be the generating series of the complex numbers $S_m$. We denote by $\frakf |_{\IH_\pm}$ the restriction of $\frakf$ to $\IH_\pm$, respectively. Then, $\frakf(y)$ is holomorphic in $\IC\setminus\IR$, and there is a canonical correspondence between the pair of Dirichlet series $L_\pm$ and the pair of $q$-series $\frakf |_{\IH_\pm}$ via Mellin transform and its inverse. 
\end{prop}
\begin{proof}
Let us denote $\alpha':=\mbox{max}(0, \alpha)$.
A simple computation shows that
\be \label{eq: mellin-formula}
    (2\pi)^{-s}\Gamma(s)L_{\pm}(s)= \int_0^\infty t^{s-1}\, \frakf(\pm \ri t) \, dt \, , \quad \Re(s)>\alpha' \, ,
\ee
whose RHS is the Mellin transform of the $q$-series $\frakf|_{\IH_\pm}$ after a change of variable.
Conversely, by the Mellin inversion theorem, we have that\footnote{The properties of the Mellin transform constrain the asymptotic behavior of the $q$-series $\frakf$. Specifically, we have that $\frakf(\pm \ri t)=O(t^{-\alpha'})$ for $t \to 0^+$ and $\frakf(\pm \ri t)=O(t^{-\infty})$ for $t \to +\infty$.}
\begin{equation} \label{eq: invmellin-formula}
\frakf(\pm \ri t)= \frac{1}{2\pi \ri}\displaystyle \int_{C-\ri\infty}^{C+\ri \infty} (2\pi t)^{-s} \Gamma(s)\, L_{\pm}(s)\, ds   \, , \quad t >0 \, , 
\end{equation}
that is, $\frakf|_{\IH_\pm}$ can be recovered via the inverse Mellin transform of the functions in Eq.~\eqref{eq: mellin-formula} for a fixed real number $C>\alpha'$. 
Finally, Eq.~\eqref{eq: invmellin-formula} can be analytically continued to $\frakf(y)$, $y \in \IC\setminus\IR$. 
\end{proof}

Note that Proposition~\ref{prop:inverse Mellin} applies when the Dirichlet series in Eq.~\eqref{eq:lseries-pm} are $L$-series. 
We can equivalently consider another pair of $L$-series, that is, 
\be
L_0(s)=\sum_{m \in \IZ_{\ne 0}} \frac{S_m}{|m|^s} \,, \quad L_1(s)=\sum_{m \in \IZ_{\ne 0}} \mathrm{sgn}(m) \frac{S_m}{|m|^s}\, ,
\ee
which are again in a one-to-one correspondence with the generating function $\frakf(y)$ in Eq.~\eqref{eq:f_Am}.
Indeed, $L_0$, $L_1$ can be written as the linear combinations
\be
\begin{aligned}
L_0(s)&=\sum_{m>0}\frac{S_m}{m^s}+\sum_{m<0}\frac{S_m}{(-m)^s}=L_{+}(s)-L_{-}(s) \, , \\
L_1(s)&=\sum_{m>0}\frac{S_m}{m^s}-\sum_{m<0}\frac{S_m}{(-m)^s}=L_{+}(s)+L_{-}(s)\,.
\end{aligned}
\ee
We can then construct the meromorphic extension of these $L$-series, or the $L$-series in Eq.~\eqref{eq:lseries-pm}, that is, the corresponding $L$-functions.
The following result states that the perturbative coefficients of the asymptotic series obtained by expanding the $q$-series $\frakf(y)$ for $y \to 0$ with $\Im(y)>0$ can be written in closed form in terms of the analytic continuation of the $L$-series $L_+(s)$. An analogous statement holds for $\Im(y)<0$.
\begin{prop}\label{prop:L-series}
Let $\{S_m\}$, $m \in \IZ_{\ne 0}$, be a sequence of complex numbers such that the Dirichlet series $L_\pm (s)$ in Eq.~\eqref{eq:lseries-pm} converge absolutely in $\{\Re (s)>\alpha\}\subset\IC$ 
and can be meromorphically continued to $\{\Re(s)< 0\}\subset\IC$ through a functional equation. 
Let 
\be
\tfrakf_\pm(y)=\sum_{n=0}^\infty b^{(\pm)}_n y^n \in \IC [\![y]\!]
\ee
be the asymptotic series\footnote{Notice that the full asymptotic expansion contains an additional Riemann integral term, which we do not consider here. See~\cite[Section~6.7 by Zagier]{zagier-appendix}.} of the generating function $\frakf(y)$, defined in Eq.~\eqref{eq:f_Am}, as $y\to 0$ with $\Im(y)>0$ and $\Im(y)<0$, respectively. Then, the perturbative coefficients $b^{(\pm)}_n$ are given by 
\be \label{eq: coeff-l1}
b^{(\pm)}_n=L_\pm(-n)\frac{(2\pi \ri)^n}{n!}\, , \quad n \in \IZ_{\ge 0} \, ,
\ee
where $L_\pm(-n)$ is defined via the functional equation.
\end{prop}
\begin{proof}
Let us take $\Im(y)>0$ and consider the $q$-series in Eq.~\eqref{eq:f_Am}. Computing the formal power series expansion of $q^m=\re^{2\pi \ri m y}$ for $y\to 0$ and $m \in \IZ_{>0}$ fixed yields
\be \label{eq:prop-f-exp}
\frakf(y) = \sum_{m=1}^{\infty} S_m \re^{2\pi \ri m y} \sim \sum_{m=1}^{\infty} \sum_{n=0}^{\infty} S_m \frac{(2\pi \ri)^n}{n!} m^n y^n \, .
\ee
Formally exchanging the order of summation, we obtain that
\be \label{eq:prop-f-exp2}
\tfrakf_+(y) = \sum_{n=0}^{\infty} \sum_{m=1}^{\infty} \frac{S_m}{m^{-n}} \frac{(2\pi \ri)^n}{n!} y^n = \sum_{n=0}^{\infty} L_+(-n) \frac{(2\pi \ri)^n}{n!} y^n \, ,
\ee 
where we make sense of $L_+(-n)$ through the meromorphic continuation of $L_+(s)$. The case of $\Im(y)<0$ is entirely analogous.
\end{proof}

Note that Proposition~\ref{prop:L-series} applies when the Dirichlet series in Eq.~\eqref{eq:lseries-pm} are $L$-functions. 
Moreover, depending on the explicit form of the functional equation satisfied by $L_\pm(s)$, it might happen that $L_\pm(-n)$ is not well-defined for some $n\in\IZ_{\ge 0}$. In this case, the computation of the corresponding perturbative coefficients $b_n^{(\pm)}$ through the formula in Eq.~\eqref{eq: coeff-l1} requires a separate ad-hoc analysis. 

\subsubsection*{The paradigm}

We can now apply the statements above to build the modular resurgence paradigm. To do so, we assemble the key ingredients in this construction and the net of relations among them along the same lines of~\cite[Section~3]{FR1phys}. Let us give a step-by-step overview below.
\begin{itemize}
\item Let $\tfrakg \in \IC [\![y]\!]$ be a Gevrey-1 formal power series whose perturbative coefficients can be written in the form of Eq.~\eqref{eq: cn-Am} with $\zeta_m=\CA m$, $m\in\IZ_{\ne 0}$, for some constant $\CA\in\IC$, and $S_m\in\IC$ such that $\sum_{m \in \IZ_{\ne 0}} |S_m m^{-n} |< \infty$ for all $n\in \IZ_{>0}$. By Proposition~\ref{prop:stokes}, $\tfrakg(y)$ is a simple resurgent series whose Borel transform has a single tower of simple poles at the locations $\zeta_m$ with Stokes constants $S_m$.
\item Assume that the complex numbers $S_m$, $m \in \IZ_{\ne 0}$, are the coefficients of two $L$-series $L_\pm(s)$, where $\{\Re(s)>\alpha\ge 0 \} \subset \IC$, defined as in Eq.~\eqref{eq:lseries-pm}. Following Proposition~\ref{prop:inverse Mellin}, we define their generating function $\frakf(y)$, $y \in \IC \setminus\IR$, in Eq.~\eqref{eq:f_Am}, by taking the inverse Mellin transform of $L_\pm(s)$ as in Eq.~\eqref{eq: invmellin-formula}. 
\item Assume that $L_\pm(s)$ can be meromorphically continued to $\{\Re(s)< 0\}\subset\IC$, \emph{i.e.}, they are $L$-functions. Denote by $\tfrakf_\pm \in \IC [\![y]\!]$ the asymptotic series of $\frakf(y)$ as $y\to 0$ with $\Im(y)>0$ and $\Im(y)<0$, respectively. Following Proposition~\ref{prop:L-series}, the perturbative coefficients of $\tfrakf_\pm(y)$ are explicitly given by the values of the $L$-functions $L_\pm(s)$ at negative integers as in Eq.~\eqref{eq: coeff-l1}, respectively. Note that $L_\pm(-n)$ for $n\in\IZ_{>0}$ are computed via the functional equation that gives the analytic continuation of the corresponding $L$-series. Sometimes, the analytic continuation of an $L$-function is described in terms of a different $L$-function with different coefficients. Assume this is the case and denote the new coefficients by $R_m \in \IC$, $m\in\IZ_{\neq 0}$. 
\item Crucially, the resurgent structures of $\tfrakf_\pm(y)$ are fully determined by the functional equation satisfied by $L_\pm(s)$, respectively. Assume that the perturbative coefficients of 
\be
\tfrakf(y)=\tfrakf_+(y)-\tfrakf_-(-y)
\ee
can be written in terms of $\eta_m$ and $R_m$ in the form of Eq.~\eqref{eq: cn-Am}, where $\eta_m=\CB m$, $m \in \IZ_{\ne 0}$, for some constant $\CB\in\IC$, while also $\sum_{m \in \IZ_{\ne 0}} |R_m m^{-n}| < \infty$ for all $n\in \IZ_{>0}$. By Proposition~\ref{prop:stokes}, $\tfrakf(y)$ is a simple resurgent series whose Borel transform has a single tower of simple poles at the locations $\eta_m$ with Stokes constants $R_m$.
\item Let $L_\pm'(s)$, where $\{\Re(s)>\beta\ge 0\} \subset \IC$, be the two $L$-series, defined as in Eq.~\eqref{eq:lseries-pm}, with coefficients $R_m$, $m \in \IZ_{\ne 0}$. Following Proposition~\ref{prop:inverse Mellin}, we can build their generating function $\frakg(y)$, $y \in \IC \setminus\IR$, defined as in Eq.~\eqref{eq:f_Am}, by taking the inverse Mellin transform of $L_\pm'(s)$ as in Eq.~\eqref{eq: invmellin-formula}.
\item Finally, $L_\pm'(s)$ are $L$-functions by definition, and therefore they can be meromorphically continued to $\{\Re(s)< 0\}\subset\IC$. Denote by $\tfrakg_\pm \in \IC [\![y]\!]$ the asymptotic series of $\frakg(y)$ as $y\to 0$ with $\Im(y)>0$ and $\Im(y)<0$, respectively. Again, following Proposition~\ref{prop:L-series}, the perturbative coefficients of $\tfrakg_\pm(y)$ are explicitly given by the values of the $L$-functions $L_\pm'(s)$ at negative integers as in Eq.~\eqref{eq: coeff-l1}, respectively. Once more, $L_\pm'(-n)$ for $n\in\IZ_{>0}$ are computed by means of the functional equation that gives the analytic continuation of the corresponding $L$-series. By construction, these are determined by $L_\pm(s)$. We recover the original formal power series as
\be
\tfrakg(y)=\tfrakg_+(y)-\tfrakg_-(-y) \, .
\ee
\end{itemize}   

Notice that the above relations among the dual resurgent structures of the pair of asymptotic series $\tfrakg$, $\tfrakf$ fit into a unique global construction that generalizes the strong-weak symmetry of~\cite{FR1phys}.
For instance, the $L$-functions $L_\pm$, $L_\pm'$ are called \emph{resurgent $L$-functions} in~\cite[Section~3]{FR1phys}.
We summarize the previous circular argument in the commutative diagram below. 
\begin{equation}\label{diag:resurgence-L funct}
\begin{tikzcd}[column sep=2.4em, row sep=2.8em]
       L_\pm'(s) \arrow[rrr,"\text{inverse Mellin}"]\arrow[ddrrrrrrrr,sloped,"\text{functional equation}"] & & & \mathfrak{g}(y) \arrow[rr,"y \rightarrow 0"] &  & \tilde{\mathfrak{g}}(y) \arrow[rrr,"\text{resurgence}"] & & & \{ S_m \}\arrow[dd,sloped,"\text{$L$-function}"] \\  \\
       \{ R_m \} \arrow[uu,sloped,"\text{$L$-function}"] & & &\tilde{\mathfrak{f}}(y) \arrow[lll,"\text{resurgence}"] 
       &  & \mathfrak{f}(y)  \arrow[ll,"y \rightarrow 0"] & & & L_\pm(s) \arrow[lll,"\text{inverse Mellin}"]\arrow[uullllllll,sloped,swap,"\text{}"]
\end{tikzcd}
\end{equation}
Here, a special place is occupied by the two-headed diagonal arrow representing the functional equation connecting the two pairs of resurgent $L$-functions. 
Moreover, the $q$-series $\frakg$, $\frakf$ are such that one equals the discontinuity of the asymptotic expansion of the other in the appropriate variables.

\begin{cor} \label{cor:disc-exchange}
Let $\tfrakg \in \IC [\![y]\!]$ be a Gevrey-1 asymptotic series whose Borel transform has a tower of simple poles at the locations $\zeta_m=\CA m$, $m\in\IZ_{\ne 0}$, for some constant $\CA\in\IC$, and Stokes constants $S_m\in\IC$.
Let $\frakf(y)$ be the generating series of the Stokes constants, as in Eq.~\eqref{eq:f_Am}.
Then, the discontinuities of $\tfrakg$ across its Stokes rays can be expressed as
\begin{subequations}
\begin{align}
\mathrm{disc}_{\theta} [\tfrakg](y) &= \sum_{m=1}^{\infty} S_m \, \re^{-\CA m /y} = \frakf\left(-\tfrac{\CA}{2 \pi \ri y}\right)   \, , \quad y\in\IH_+ \cap \{\Re ( \re^{-\ri\theta} y)>0\} \, ,   \label{eq: disc-up}\\
\mathrm{disc}_{\theta+\pi} [\tfrakg ] (y) &=\sum_{m=1}^{\infty} S_{-m} \, \re^{\CA m /y} = -\frakf\left(-\tfrac{\CA}{2 \pi \ri y}\right)   \, , \quad y\in\IH_- \cap \{\Re ( \re^{-\ri\theta} y)>0\} \, , \label{eq: disc-down}
\end{align}
\end{subequations}
where $0 \le \theta < \pi$ is the argument of the singularities in the upper half of the Borel plane.
\end{cor}
\begin{proof}
The statement follows from Eqs.~\eqref{eq: Stokes1-poles} and~\eqref{eq:f_Am}.
\end{proof}

\begin{rmk}
    From the conventional point of view in resurgence, considering the generating series of the Stokes constants is rather unusual, as one would not generally expect them to define a function. However, we argue here that when the resurgent structure of an asymptotic series is of a particularly simple type, while its Stokes constants have special analytic and number-theoretic properties, studying the generating series of the Stokes constants allows us to construct a new asymptotic series whose resurgence is deeply related to that of the original one. This is the domain of application of modular resurgence, as described above. We are therefore prompted to advocate a novel perspective on resurgence that investigates the properties of the Stokes constants in their own right.
\end{rmk}

\subsection{Modular resurgent structures} \label{sec: modular-structures}

Our previous discussion highlights the leading role of the $L$-functions and their functional equation in describing new features of a particular type of resurgent asymptotic series. In this section, we address the relationship between quantum modularity and resurgent structures, which will clarify our choice of name for the paradigm of modular resurgence. 
We start by proposing the following definitions of modular resurgent structure and modular resurgent series (MRS).
\begin{definition}\label{def:modular_res_struct}
A Gevrey-1 asymptotic series $\tfrakg \in \IC \llbracket y\rrbracket$ has a \emph{modular resurgent structure} if the following conditions hold.
\begin{enumerate}
    \item The Borel transform $\CB[\tfrakg] \in \IC\{\zeta\}$ has a tower of singularities at the locations $\zeta_m= \CA m$, $m\in\mathbb{Z}_{\ne 0}$, for some constant $\CA \in \IC$, in the complex $\zeta$-plane.
    \item For every $m\in\IZ_{\ne 0}$, the resurgent series at the singularity $\zeta_m$ is the constant function $S_m\in\IC$, \emph{i.e.}, the Stokes constant.
    \item The Stokes constants $S_m$, $m\in\IZ_{\ne 0}$, are the coefficients of two $L$-functions
    \be
		L_{+}(s)=\sum_{m >0} \frac{S_{m}}{m^s} \,, \quad L_{-}(s)=-\sum_{m >0} \frac{S_{- m}}{m^s} \, .
     \ee
\end{enumerate}
A Gevrey-1 asymptotic series with a modular resurgent structure is a \emph{modular resurgent series}. 
\end{definition}

Note that in the definition of a modular resurgent structure, we have left unspecified the type of singularities appearing in the Borel plane. However, their nature directly affects the expression of the coefficients of the asymptotic series $\tfrakg(y)$. For instance, when all singularities in the tower are simple poles, as in Propositions~\ref{prop:stokes} and~\ref{prop:stokes2}, the perturbative coefficients depend on the Stokes data via the exact large-order relations in Eq.~\eqref{eq: cn-Am}. Recall that this assumption is made throughout Section~\ref{sec:new-resur-Lfunct}. In fact, the modular resurgence paradigm, as discussed before, applies to a \emph{canonical} pair of MRSs $\tfrakg(y)$ and $\tfrakf(y)$ whose Borel transforms possess a tower of simple poles and such that one reproduces the asymptotic expansion of the discontinuity of the other (see Corollary~\ref{cor:disc-exchange}). The paradigm thus formalizes an exchange of perturbative and non-perturbative information between opposite asymptotic regimes in the same spirit as strong-weak dualities in physical theories. Notice also that not all MRSs fulfill the criteria for the paradigm to apply.

\subsubsection*{The conjectures}

Let us now state our main conjectures. Specifically, a $q$-series with modular resurgent asymptotics is expected to be a holomorphic quantum modular form and, if it is the inverse Mellin transform of an $L$-function, to agree with the median resummation of its asymptotic expansion.

\begin{conjecture}\label{conj:quantum_modular1}
Let $\frakg\colon\IH\to\IC$ be a $q$-series where $q=\re^{2 \pi \ri y}$. If its Mellin transform is an $L$-function and its asymptotic expansion $\tfrakg(y)$ as $y\to 0$ with $\Im(y)>0$ has a modular resurgent structure, then the \emph{median resummation} of $\tfrakg(y)$ reconstructs the original function $\frakg(y)$, that is, 
\be
\mathcal{S}_\theta^{\mathrm{med}}[\tfrakg](y)=\frakg(y) \, , \quad y\in\IH \cap \{\Re ( \re^{-\ri\theta} y)>0\} \, ,
\ee
where $\theta$ is the argument of the singularities in the Borel plane. 
\end{conjecture}
\begin{conjecture}\label{conj:quantum_modular2}
Let $\frakg\colon\IH\to\IC$ be a $q$-series where $q=\re^{2 \pi \ri y}$. If its asymptotic expansion $\tfrakg(y)$ as $y\to 0$ with $\Im(y)>0$ has a modular resurgent structure, then the function $\frakg(y)$ is a \emph{holomorphic quantum modular form} for a subgroup $\Gamma\subseteq\mathsf{SL}_2(\mathbb{Z})$.
\end{conjecture} 

We can attempt to change the assumptions in Conjectures~\ref{conj:quantum_modular1} and~\ref{conj:quantum_modular2} along the same lines as in Zagier's strange identity~\cite{zagier2001vassiliev} and consider functions that are defined only over the rationals. This is, in fact, the setting where quantum modular forms were originally introduced~\cite{zagier_modular}, as briefly discussed in Section~\ref{sec:modularity}. Conjectures~\ref{conj:quantum_modular1} and~\ref{conj:quantum_modular2} then become the following.
\begin{conjecture}\label{conj:quantum_modular1-Q}
Let $\frakg\colon\IQ\to\IC$ be a $q$-series where $q=\re^{2 \pi \ri y}$. If its Mellin transform is an $L$-function and its asymptotic expansion $\tfrakg(y)$ as $y\to \ri 0$
extends to $y \in \IC$ and has a modular resurgent structure, then the \emph{median resummation} of $\tfrakg(y)$ reconstructs the original function $\frakg(y)$, that is,
\be \label{eq:median-Q}
\mathcal{S}_\theta^{\mathrm{med}}[\tfrakg](y)=\frakg(y) \, , \quad y\in\IQ \cap \{\Re ( \re^{-\ri\theta} y )>0\} \, , 
\ee
where $\theta$ is the argument of the singularities in the Borel plane. 
\end{conjecture}
\begin{conjecture}\label{conj:quantum_modular2-Q}
Let $\frakg\colon\IQ\to\IC$ be a $q$-series where $q=\re^{2 \pi \ri y}$. If its asymptotic expansion $\tfrakg(y)$ as $y\to \ri 0$ extends to $y \in \IC$ and has a modular resurgent structure, then the function $\frakg(y)$ is a \emph{quantum modular form} for a subgroup $\Gamma\subseteq\mathsf{SL}_2(\mathbb{Z})$.
\end{conjecture}

We stress that to perform a resurgent analysis of the formal power series $\tfrakg(y)$ appearing in Conjectures~\ref{conj:quantum_modular1-Q} and~\ref{conj:quantum_modular2-Q}, it is necessary to complexify the variable $y$. Namely, $\tfrakg(y)$ must be analytically continued to $y \in \IC$. In Conjecture~\ref{conj:quantum_modular1-Q}, the median resummation $\mathcal{S}_\theta^{\mathrm{med}}[\tfrakg](y)$ is then restricted to the original domain of the function $\frakg(y)$ to make sense of Eq.~\eqref{eq:median-Q}.

\begin{rmk}
Not all (holomorphic) quantum modular forms give rise to MRSs. Indeed, by construction, those whose cocycle $h_\gamma$ has a convergent asymptotic expansion at $\gamma^{-1}(\infty)$ for all $\gamma \in \Gamma$ cannot be described via modular resurgence.
For instance, this is the case of the function $\tilde{\Delta}$ defined in Eq.~\eqref{eq:cocycle-Delta}, whose cocycle is a polynomial---hence, its resurgent structure is empty.
\end{rmk} 

Note that the generating function of the Stokes constants in Eq.~\eqref{eq:f_Am} is supported on $\IC\setminus\IR$. Indeed, all statements above can be recast for a $q$-series $\frakg\colon\IH_{-}\to\IC$, in which case the relevant asymptotic expansion is obtained in the limit $y\to 0$ with $\Im(y)<0$. 

As we discuss in Section~\ref{sec:examples}, Conjectures~\ref{conj:quantum_modular1} and~\ref{conj:quantum_modular2} are verified by the generating functions of the weak and strong coupling Stokes constants for $\log \mathrm{Tr}(\mrho_{\IP^2})$~\cite{Rella22, FR1phys} (see Section~\ref{sec:example-P2}) and a family of examples built from the theory of Maass cusp forms and their periods (see Theorems~\ref{thm:lewis-Zagier_QM} and~\ref{thm:lewis-Zagier_median}). Besides, if we relax the assumption that $\frakg$ must be a $q$-series, then MRSs satisfying our conjectures can be built from general cusp forms (see Theorem~\ref{thm:modular}).
Evidence in support of Conjectures~\ref{conj:quantum_modular1-Q} and~\ref{conj:quantum_modular2-Q} is given by the Kontsevich--Zagier series\footnote{In the case of the Kontsevich--Zagier series, the resurgent structure has a half-tower of fractional power singularities, and not simple poles, contrary to what occurs in the examples discussed in this paper.}~\cite{costin-garoufalidis, zagier2001vassiliev}, the WRT invariants for a family of Seifert fibered homology spheres~\cite[Theorem~3]{qCS7}, the LMO ($\hat{Z}$) invariants for a family of plumbed 3-manifolds~\cite{BKM1,BKM2}, and conjecturally the $q$-series $\sigma$ and $\sigma^*$ in~\cite[Example~1]{zagier_modular}.
According to~\cite[Conjecture~1.4]{costin-garoufalidis}, yet more instances of Conjecture~\ref{conj:quantum_modular1-Q} might appear in quantum topology. We leave a more comprehensive overview of current evidence for future work.

\subsubsection*{The role of $L$-functions}

As we have already pointed out, the modular resurgence paradigm of Section~\ref{sec:new-resur-Lfunct} is essentially governed by the functional equation for the $L$-functions defined by the Stokes constants. 
However, at this stage, more assumptions may be needed to fully characterize these resurgent $L$-functions and potentially enable us to prove the modular resurgence conjectures of Section~\ref{sec: modular-structures}. 
Let us collect here some considerations in this direction.

Firstly, Proposition~\ref{prop:L-series} implies that the asymptotic series of the generating function of the Stokes constants is dictated by the analytic continuation of the corresponding $L$-function. This is the core of the modular resurgence paradigm represented by the commutative diagram in Eq.~\eqref{diag:resurgence-L funct}.
In some examples, as in the case of Eq.~\eqref{eq:functional-Lambda}, the functional equation for $L_{\pm}(s)$ in Eq.~\eqref{eq:lseries-pm} is written in terms of the same $L$-functions.
Equivalently, it might occur that the asymptotic expansion in the limit $y \to 0$ with $\Im(y)>0$ (respectively, $\Im(y)<0$) of the generating function of the Stokes constants $\{S_m\}$, $m \in \IZ_{\ne 0}$, which are obtained from the resurgence of a MRS $\tfrakg(y)$, reproduces the same $\tfrakg_{+}(y)$ (respectively, $\tfrakg_{-}(y)$). In this case, the diagram in Eq.~\eqref{diag:resurgence-L funct} reduces to the one below.
\begin{equation} \label{diag:simple resur}
\begin{tikzcd}[column sep=2.4em, row sep=2.8em]
\arrow[loop left]{l}{\substack{\text{functional} \\ \text{equation}}} L_{\pm}(s)\arrow[rrr, "\text{inverse Mellin}"] & & & \frakg(y) \arrow[rr,"y\to 0"] & & \tfrakg(y)  \arrow[rrr,"\text{resurgence}"] & & & \{ S_m \}\arrow[bend left=15]{llllllll}{\text{$L$-function}} 
\end{tikzcd}
\end{equation}
This simplified version of the modular resurgence paradigm is observed, among others, in the examples of the series $\sigma$ and $\sigma^*$, whose $L$-function was first studied by Cohen in~\cite{cohen1988q}.\footnote{The study of the resurgent structure of $\sigma$ and $\sigma^*$ was presented in~\cite[from min.~21.55]{Fantini-talk-IHES}. 
} More generally, a similar behavior is observed when $\frakg$ is the function associated with a Maass cusp form of the kind considered in~\cite{lewis-Zagier--period}. See Proposition~\ref{prop:lewis-Zagier-qms}.  

Secondly, we might ask how to distinguish which $L$-functions can appear within the modular resurgence framework and which cannot. For instance, can $L$-functions coming from modular forms occur? 
As we will discuss in Section~\ref{sec:modular}, when $L_g(s)$ is the $L$-function associated with a general cusp form $g: \IH \to \IC$, we can take the generating function in Eq.~\eqref{eq:f_Am} to be the Fourier series of $g(y)$ itself. However, the asymptotic expansion of the cusp form for $y \to 0$ gives a convergent power series---in other words, the resurgent structure is trivial. An additional step is then warranted, although at the cost of breaking modular invariance. Namely, we will argue that adding to $g$ an appropriate correction that is uniquely and explicitly determined by its Fourier coefficients, we can construct a new function $f: \IH \to \IC$, and expanding the latter in the limit $y \to 0$ yields a Gevrey-1 asymptotic series $\tilde{f} \in \IC[\![y]\!]$. In this case, the simplified modular resurgence diagram in Eq.~\eqref{diag:simple resur} does not close, as pictured below.
\begin{equation} \label{diag:simple resur-open}
\begin{tikzcd}[column sep=2.4em, row sep=2.8em]
L_{g}(s)\arrow[rrr, "\text{inverse Mellin}"] & & & g(y) \arrow[rr, dashed, red] & & f(y) \arrow[rr,"y\to 0"] & & \tilde{f}(y)  \arrow[rrr,"\text{resurgence}"] & & & \{ A_m \}\arrow[bend left=15]{llllllllll}{\text{$L$-function}} 
\end{tikzcd}
\end{equation}
This implies that $L$-functions coming from general cusp forms cannot appear in the original version of the modular resurgence paradigm.
Nonetheless, as we prove in Section~\ref{sec:modular}, the above formal power series $\tilde{f}$ fits the definition of a modular resurgent series and the function $f$ satisfies the theses of Conjectures~\ref{conj:quantum_modular1} and~\ref{conj:quantum_modular2} despite not being a $q$-series. 

\section{Examples of modular resurgent series} \label{sec:examples}
In this section, we describe two complete examples of the paradigm of modular resurgence. First, we show how the exact resurgent symmetry of Section~\ref{sec:new-resur-Lfunct} reproduces the strong-weak symmetry discovered in~\cite[Section~3]{FR1phys} once specialized to the spectral trace of local $\IP^2$. Then, we prove the conjectural statements brought forward in Section~\ref{sec: modular-structures} for a large class of modular resurgent series originating from the theory of Maass cusp forms. Finally, we present an example of modular resurgent structure built from general cusp forms.

\subsection{From topological strings} \label{sec:example-P2}
Let us compare the paradigm of modular resurgence with the \emph{strong-weak resurgent symmetry} of the spectral trace of local $\IP^2$ and show that the construction presented here and the one of~\cite[Section~3]{FR1phys} are coincident and complementary.

We start by recalling the necessary background from~\cite{Rella22, FR1phys}. Following the prescription of~\cite{GHM}, quantizing the mirror curve to local $\IP^2$ gives rise to a quantum operator acting on $L^2(\IR)$ whose inverse $\rho_{\IP^2}$ is positive-definite and of trace class. The first fermionic spectral trace\footnote{The fermionic spectral traces $Z_X(\bm{N},\hbar)$, $\bm{N} \in \IZ^{g_X}$, provide a non-perturbative definition of topological string theory on a toric CY threefold $X$ according to the topological string/spectral theory (TS/ST) correspondence~\cite{GHM, CGM2}. Here, $g_X$ is the genus of the mirror curve to $X$. The Planck constant $\hbar$ satisfies a strong-weak coupling duality $\hbar \propto g_s^{-1}$, where $g_s$ is the string coupling constant.} $Z_{\IP^2}(1, \hbar)=\mathrm{Tr}(\mrho_{\IP^2})$ is a well-defined analytic function of the quantum deformation parameter $\hbar \in \IC'$, which is known in closed form for $\Im(\hbar)>0$~\cite{KM}.
The all-orders perturbative expansions of the logarithm of the spectral trace in the limits of $\hbar \rightarrow 0$ and $\hbar \rightarrow \infty$ have been computed by the second author in~\cite{Rella22}. They give rise to the Gevrey-1 asymptotic series 
\begin{subequations}
\begin{align}
\phi(\hbar) &= \sum_{n=1}^{\infty} a_{2n} \hbar^{2n} \in \IQ[\![\hbar]\!] \, , \label{eq: phiP2} \\
\psi(\tau) &= \sum_{n=1}^{\infty} b_{2n} \tau^{2n-1} \in \IQ[\pi, \sqrt{3}] [\![\tau]\!] \, ,\quad \tau =-\tfrac{2 \pi}{3 \hbar} \, , \label{eq: phiP2infty}
\end{align}
\end{subequations}
whose Borel transforms possess a single tower of simple poles at the locations 
\be
\zeta_m= \CA_0 \ri m \, , \quad \eta_m= \CA_\infty \ri m \, , \quad m \in \IZ_{\ne 0} \, ,
\ee
where $\CA_0=\frac{4 \pi^2}{3}$ and $\CA_\infty=\frac{2 \pi}{3}$, respectively. The corresponding Stokes constants
\be
S_m = S_{-m} \in \sqrt{3} \ri \, \IQ_{>0} \, , \quad R_m =-R_{-m} \in \IQ_{\ne 0} \, , \quad m \in \IZ_{\ne 0} \, ,
\ee
are expressed explicitly as divisor sum functions and define the weak and strong coupling $L$-series
\be \label{eq:L-functP2}
L_0(s) = \sum_{m=1}^\infty\frac{S_m}{m^s} \, , \quad L_\infty(s) = \sum_{m=1}^\infty\frac{R_m}{m^s} \, ,
\ee
which converge in the right half-plane $\Re(s)>1$ and can be analytically continued to meromorphic functions on the complex $s$-plane~\cite[Section~4.4]{Rella22}. Therefore, $L_0$ and $L_\infty$ are $L$-functions, while $\phi$ and $\psi$ are MRSs according to Definition~\ref{def:modular_res_struct}.

In the companion paper~\cite{FR1phys}, we complete the resurgent $L$-functions in Eq.~\eqref{eq:L-functP2} into 
\begin{subequations}
\begin{align}
    \Lambda_0(s) &= \frac{3^{\frac{s}{2}-1}}{\ri \pi^{s+1}} \Gamma\left( \frac{s}{2} \right) \Gamma\left( \frac{s}{2}+1 \right) L_0(s) \, , \label{eq: defLambda0} \\
    \Lambda_\infty(s) &= \frac{3^{\frac{s}{2}-1}}{\pi^{s+1}} \Gamma\left( \frac{s+1}{2} \right)^2 L_\infty(s) \,  \label{eq: defLambdainfty} 
\end{align}
\end{subequations}
and show that $\Lambda_0$ and $\Lambda_\infty$ can be analytically continued to $\Re(s)<0$ through each other. In particular, they obey the combined functional equation
\be \label{eq: duality}
\Lambda_0(s) = \Lambda_\infty(-s) \, ,
\ee
which lies at the core of the exchange of information captured by the modular resurgence paradigm.
Simultaneously, we take the perspective of the generating functions $f_0(y)$ and $f_\infty(y)$, $y \in \IC\setminus\IR$, of the weak and strong coupling Stokes constants $\{S_m\}$ and $\{R_m\}$, $ m\in\IZ_{>0}$, which satisfy the parity properties
\be
f_0(-y)=-f_0(y) \, , \quad f_\infty(-y)=f_\infty(y) \, .
\ee
Their all-orders perturbative expansions in the limit $y \rightarrow 0$ with $\Im(y)>0$ are
\begin{subequations}
\begin{align}
\tilde{f}_0(y) &= -\frac{\pi \ri}{2} - \frac{3 \mathcal{V}}{2 \pi \ri y} - 2 \psi(y) \, , \label{eq: f0-psi}\\
\tilde{f}_\infty(y) &= - 3 \log \frac{\Gamma(2/3)}{\Gamma(1/3)} -\log(- 6 \pi \ri y) +2 \phi(2 \pi y) \, , \label{eq: finf-phi}
\end{align}
\end{subequations}
where $\mathcal{V}= 2 \Im\left(\mathrm{Li}_2(\re^{2 \pi \ri/3}) \right)$.

In the following, let us neglect the terms that do not enter the perturbative series in the RHS of Eqs.~\eqref{eq: f0-psi} and~\eqref{eq: finf-phi}. We will show that $\tilde{f}_0$ and $\tilde{f}_\infty$ fit into the newly introduced paradigm of modular resurgence. This offers a supplementary understanding of the strong-weak resurgent symmetry discussed in~\cite[Section~3]{FR1phys}, particularly highlighting the central role played by the functional equation in Eq.~\eqref{eq: duality}. 
For completeness, we retrace each step in the paradigm for this key example. 
\begin{itemize}
 \item The resurgent structure of $\tilde{f}_{\infty}(y)$ is dictated by the formal power series $2 \phi(2 \pi y)$ according to Eq.~\eqref{eq: finf-phi} and is therefore modular. We write it for simplicity as $\{\eta_m, \, S_m\}$, $m\in\IZ_{\neq 0}$, where the Stokes constants satisfy $S_{-m}=S_m$. 
 \item The resurgent $L$-series $L_0(s)$, $s \in \IC$ with $\Re(s)>1$, has coefficients given by the Stokes constants $S_m$ for $m \in \IZ_{>0}$. Note that $L_0(s) = \pm L_{0, \pm}(s)$ in the notation of Eq.~\eqref{eq:lseries-pm}. 
\item The generating function $f_0$ is given by the inverse Mellin transform of $L_0$ as in Eq.~\eqref{eq: invmellin-formula}. Let 
\be
\tilde{f}_0(y)=\sum_{n=1}^\infty c_n^0 y^n\in\IC\llbracket y\rrbracket 
\ee
be its asymptotic series as $y\to 0$ with $\Im(y) >0$. By Proposition~\ref{prop:L-series}, particularly Eq.~\eqref{eq: coeff-l1}, the perturbative coefficients $c_n^{0}$, $n \in \IZ_{> 0}$, satisfy
\be \label{eq: lemma17P2}
    c_n^{0}=L_0(-n) \frac{(2\pi \ri)^{n}}{n!} \, , 
\ee
where we make sense of the RHS via the functional equation for the completed $L$-function $\Lambda_0$. In particular, substituting Eqs.~\eqref{eq: defLambda0} and~\eqref{eq: defLambdainfty} into Eq.~\eqref{eq: duality}, we find that
\be \label{eq: L0-cont}
    L_0(-s)=-\frac{2 \ri 3^s}{\pi^{2s} s}\frac{\Gamma\left(\frac{s+1}{2}\right)^2}{\Gamma\left(-\frac{s}{2}\right)^2} L_\infty(s) \, .
\ee
Notice that $L_0(-s)$ tends to zero for $s=n\in\IZ_{> 0}$ even as a consequence of the divergence of the gamma function in the denominator. Thus, we set $c_{2n}^0=0$, $n\in\IZ_{> 0}$.
Then, taking $s= n \in \IZ_{> 0}$ odd, applying the well-known properties of the gamma function, and using the series representation for $L_\infty$ in Eq.~\eqref{eq:L-functP2}, Eq.~\eqref{eq: L0-cont} becomes
\be
\begin{aligned}
L_0(-n)&=-\frac{2\ri}{\pi} n!(n-1)! (-2 \pi \ri)^{-n} \sum_{m=1}^\infty\frac{R_m}{\eta_m^n}\, . 
\end{aligned}
\ee
Applying Eq.~\eqref{eq: lemma17P2}, the non-trivial perturbative coefficients $c_{2n+1}^0$, $n \in \IZ_{\ge 0}$, are
\be
    c_{2n+1}^{0}=-2 \frac{\Gamma(2n+1)}{\pi \ri} \sum_{m=1}^\infty\frac{R_m}{\eta_m^{2n+1}} = -2 b_{2n+2} \, ,
\ee
where we have used the exact large-order identity in Eq.~\eqref{eq: cn-Am} for the coefficients $\{b_{2k}\}$, $k \in \IZ_{>0}$, of the formal power series $\psi$ in Eq.~\eqref{eq: phiP2infty}. This is in agreement with Eq.~\eqref{eq: f0-psi}. 
\item The resurgent structure of $\tilde{f}_0(y)$ is dictated by the formal power series $-2\psi(y)$ according to Eq.~\eqref{eq: f0-psi} and is therefore modular. We write it for simplicity as $\{\eta_m, R_m\}$, $m\in\IZ_{\neq 0}$, where the Stokes constants satisfy $R_{-m}=-R_m$.
\item The resurgent $L$-series $L_\infty(s)$, $s \in \IC$ with $\Re(s)>1$, has coefficients given by the Stokes constants $R_m$ for $m \in \IZ_{>0}$. Note that $L_\infty(s) = L_{\infty, \pm}(s)$ in the notation of Eq.~\eqref{eq:lseries-pm}. 
\item The generating function $f_\infty$ is given by the inverse Mellin transform of $L_\infty$ as in Eq.~\eqref{eq: invmellin-formula}. Let 
\be
\tilde{f}_\infty(y)=\sum_{n=1}^\infty c_n^\infty y^n\in\IC\llbracket y\rrbracket
\ee
be its asymptotic series as $y\to 0$ with $\Im(y)>0$. By Proposition~\ref{prop:L-series}, particularly Eq.~\eqref{eq: coeff-l1}, its perturbative coefficients $c_n^{\infty}$, $n \in \IZ_{> 0}$, satisfy
\be \label{eq: lemma17P2-2}
    c_n^{\infty}=L_\infty(-n) \frac{(2\pi \ri)^{n}}{n!} \, , 
\ee
where we make sense of the RHS via the functional equation for the completed $L$-function $\Lambda_\infty$. In particular, Eq.~\eqref{eq: L0-cont} gives
\be \label{eq: Linf-cont}
    L_\infty(-s)=\frac{3^{s} s}{2 \ri \pi^{2s}}\frac{\Gamma\left(\frac{s}{2}\right)^2}{\Gamma\left(\frac{1-s}{2}\right)^2} L_0(s) \, .
\ee
Notice that $L_\infty(-s)$ tends to zero for $s=n\in\IZ_{\ge 0}$ odd due to the divergence of the gamma function in the denominator.
Hence, we set $c_{2n+1}^\infty=0$, $n\in\IZ_{\ge 0}$.
Then, taking $s= n \in \IZ_{>0}$ even, applying the well-known properties of the gamma function, and using the series representation for $L_0$ in Eq.~\eqref{eq:L-functP2}, Eq.~\eqref{eq: Linf-cont} becomes
\be
\begin{aligned}
L_\infty(-n)&=\frac{2}{\ri \pi} n!(n-1)! (-2 \pi)^{n} \sum_{m=1}^\infty\frac{S_m}{\zeta_m^n}\, . 
\end{aligned}
\ee
Applying Eq.~\eqref{eq: lemma17P2-2}, the non-trivial perturbative coefficients $c_{2n}^\infty$, $n \in \IZ_{>0}$, are 
\be
    c_{2n}^{\infty}=2 (2 \pi)^{2n} \frac{\Gamma(2n)}{\pi \ri} \sum_{m=1}^\infty\frac{S_m}{\zeta_m^{2n}} = 2 (2 \pi)^{2n} a_{2n} \, ,
\ee
where we have used the exact large-order identity in Eq.~\eqref{eq: cn-Am} for the coefficients $\{a_{2k}\}$, $k \in \IZ_{>0}$, of the formal power series $\phi$ in Eq.~\eqref{eq: phiP2}. Again, this agrees with Eq.~\eqref{eq: finf-phi}.
\end{itemize}
\begin{rmk}
    In the discussion above, we have not considered the term of order $y^0$ in our definitions of $\tilde{f}_0(y)$ and $\tilde{f}_\infty(y)$. We observe, however, that direct evaluation of $L_0(s)$ at $s=0$ using the closed formula in~\cite[Eq.~(4.174a)]{Rella22} together with Eq.~\eqref{eq: lemma17P2} gives the perturbative coefficient 
    \be
    c_0^0=L_0(0)=-\frac{\pi \ri}{2} \, , 
    \ee
    which agrees with Eq.~\eqref{eq: f0-psi}. At the same time, the formula for $L_\infty(s)$ in~\cite[Eq.~(4.174b)]{Rella22} is not well-defined at $s=0$---as it is also observed from Eq.~\eqref{eq: Linf-cont}. In particular, following Eq.~\eqref{eq: lemma17P2-2}, the perturbative coefficient $c_0^\infty$ would appear to diverge to infinity. We conclude that the constant term in Eq.~\eqref{eq: finf-phi} cannot be consistently derived from the functional equation for $L_\infty$.
\end{rmk}

We have shown that the weak and strong coupling resurgent structures of the spectral trace of local $\IP^2$ are a modular resurgent pair and satisfy the global symmetry of Section~\ref{sec:new-resur-Lfunct}. Remarkably, as we argue in~\cite{FR1phys}, this construction is responsible for the exchange of the perturbative/non-perturbative contributions to the holomorphic and anti-holomorphic blocks in the well-known factorization of the spectral trace, which can be intuitively traced back to the duality between the worldsheet and WKB contributions to the total grand potential of the topological string on local $\IP^2$.

Finally, we stress that both Conjectures~\ref{conj:quantum_modular1} and~\ref{conj:quantum_modular2} are explicitly proven for the generating function $f_0$.
Namely, $f_0$ is a holomorphic quantum modular function for the congruence subgroup $\Gamma_1(3) \subset \mathsf{SL}_2(\IZ)$~\cite[Theorem~4.6]{FR1phys} and is reconstructed by the median resummation of its asymptotic expansion~\cite[Theorem~4.7]{FR1phys}.
As for the generating function $f_\infty$, Conjecture~\ref{conj:quantum_modular2} is proven with respect to the same group $\Gamma_1(3)$~\cite[Theorem~4.7]{FR1phys}, while Conjecture~\ref{conj:quantum_modular1} is only supported by numerical evidence~\cite[Conjecture~1]{FR1phys}.

\subsection{From Maass cusp forms} \label{sec:example-Maass}
A large class of examples of MRSs can be derived from a pair of $L$-functions satisfying a special type of functional equation and deeply related to \emph{Maass cusp forms} with spectral parameter $1/2$.
Recall that a Maass cusp form with spectral parameter $\mu\in\IC$ is a smooth $\mathsf{PSL}_2(\IZ)$-invariant function $u\colon\IH\to\IC$ that satisfies the following properties:
\begin{enumerate}
    \item $u$ is an eigenfunction of the hyperbolic Laplacian $\Delta$ with eigenvalue $\mu(1-\mu)$; 
    \item $u(y)=O(\Im(y)^C)$ as $\Im(y) \to +\infty$ for some constant $C \in \IR$.
\end{enumerate}
These functions represent a basis for $L^2$ on the modular curve $\mathsf{PSL}_2(\IZ)\setminus \IH$ and are in one-to-one correspondence with pairs of Dirichlet series $L_\epsilon$, $\epsilon=0,1$, that converge in some right half-plane of $\IC$ and whose completed $L$-functions $\Lambda_\epsilon(s)=\gamma_\mu(s+\epsilon) L_\epsilon(s)$, where
\be 
\gamma_\mu(s)=\frac{1}{4\pi^s}\Gamma\Big(\tfrac{s-\mu+1/2}{2}\Big) \Gamma\Big(\tfrac{s+\mu-1/2}{2}\Big)\,,
\ee
are entire functions of finite order and satisfy the functional equations in Eq.~\eqref{eq:functional-Lambda}~\cite{Maass}.
It was shown in~\cite{lewis-Zagier--period} that there are two alternative equivalent descriptions of Maass cusp forms in terms of holomorphic functions on either $\IC\setminus \IR$ or $\IC'$ with specific features. 
Building on the results of~\cite{lewis-Zagier--period}, we consider here the $L$-functions corresponding to Maass cusp forms with $\mu=1/2$.

\begin{prop}\label{prop:lewis-Zagier-qms}
Let $\{A_m\}$, $m \in \IZ_{\ne 0}$, be a sequence of complex numbers such that the Dirichlet series 
\be \label{eq: lseries-LZ}
L_\epsilon(s):=\sum_{m=1}^\infty \frac{A_m+(-1)^\epsilon A_{-m}}{m^s} \, , \quad \epsilon=0,1 \, , 
\ee
are absolutely convergent in some right half-plane of $\IC$ and the corresponding completed $L$-functions   
\be \label{eq: Lambdaeps-gammafact} \Lambda_\epsilon(s):=\gamma(s+\epsilon)L_\epsilon(s) \, , \quad \text{where} \quad \gamma(s)=\frac{1}{4\pi^s}\, \Gamma\big(\tfrac{s}{2}\big)^2\,,
\ee
are entire functions of finite order and satisfy the functional equations
\begin{equation}\label{eq:functional-Lambda}
    \Lambda_\epsilon(1-s)=(-1)^\epsilon \Lambda_\epsilon(s) \, .
\end{equation}
Let $f\colon\IC\setminus\IR \to \IC$ be the generating function
\begin{equation}\label{eq: LZ-f}
    f(y):=\begin{cases}
    \displaystyle\sum_{m>0}A_m \re^{2\pi \ri my} & \mbox{if} \quad \Im(y)>0  \\
    & \\
    -\displaystyle\sum_{m<0}A_m \re^{2\pi \ri my} & \mbox{if} \quad \Im(y)<0 
\end{cases} \, ,
\end{equation} 
and $\tilde{f}_{+}(y)$ and $\tilde{f}_-(y)$ be its asymptotic series as $y\to 0$ with $\Im(y)>0$ and $\Im(y)<0$, respectively. Then, $\tilde{f}_\pm$ are modular resurgent series. 
\end{prop}
\begin{proof}    
    Let us introduce the two $L$-functions
    \begin{subequations}
        \begin{align}
            L_+(s)&= \sum_{m=1}^\infty \frac{A_m}{m^s}=\frac{L_{0}(s)+L_1(s)}{2}  \, , \label{eq: L0-L1-L-def} \\
            L_-(s)&= -\sum_{m=1}^\infty \frac{A_{-m}}{m^s}=\frac{L_{1}(s)-L_0(s)}{2} \, , \label{eq: L0-L1-Lminus-def} 
        \end{align}
    \end{subequations}
    as in Eq.~\eqref{eq:lseries-pm}. It follows from Proposition~\ref{prop:L-series}, particularly from Eq.~\eqref{eq: coeff-l1}, that the asymptotic series of $f(y)$ as $y\to 0$ with $\Im( y)>0$ can be written as 
    \be \label{eq: ftilde-L-proof}
    \tilde{f}_{+}(y)=\sum_{n=0}^\infty L_+(-n)\frac{(2\pi \ri )^n}{ n!}y^n \, .
    \ee
    Similarly, its asymptotic series as $y\to 0$ with $\Im(y)<0$ is 
    \be \label{eq: ftilde-Lminus-proof}
    \tilde{f}_{-}(y)=\sum_{n=0}^\infty L_{-}(-n)\frac{(2\pi \ri )^n}{ n!}y^n \, .
    \ee
    We can now apply the functional equation for $\Lambda_\epsilon(s)$ in Eq.~\eqref{eq:functional-Lambda} to make sense of $L_\pm(-n)$ for $n \in \IZ_{\ge 0}$. 
    To do so, let us consider separately the case of $n \in \IZ_{\ge 0}$ even (respectively, odd) for $L_1(-n)$ (respectively, $L_0(-n)$) due to the expression for the gamma factor $\gamma(s)$.
    For $n \in \IZ_{\ge 0}$ even, we have that 
    \be \label{eq: L1-n}
    \begin{aligned}
        L_1(-n)&=-\frac{\gamma(n+2)}{\gamma(1-n)} L_1(n+1) =-\frac{2^{-2n}}{\pi^{2n+2}} \Gamma(n+1)^2 L_1(n+1) \, ,
    \end{aligned}
    \ee
    where we have used the definition for the gamma factor in Eq.~\eqref{eq: Lambdaeps-gammafact} and applied the well-known identities for the gamma function.
    Similarly, for $n \in \IZ_{> 0}$ odd, we have that
    \be \label{eq: L0-n}
    \begin{aligned}
        L_0(-n)&=\frac{\gamma(n+1)}{\gamma(-n)} L_0(n+1) =\frac{2^{-2n}}{\pi^{2n+2}} \Gamma(n+1)^2 L_0(n+1) \, .
    \end{aligned}
    \ee
    Since the gamma function has simple poles at the non-positive integers, we set $L_1(-n)=0$ for $n \in \IZ_{> 0}$ odd and $L_0(-n)=0$ for $n \in \IZ_{\ge 0}$ even.
    Therefore, putting together Eqs.~\eqref{eq: L0-L1-L-def},~\eqref{eq: L1-n}, and~\eqref{eq: L0-n}, we find that 
    \be \label{eq: L-n}
    \begin{aligned}
        \dfrac{1}{n!}L_+(-n)&=\begin{cases} -\dfrac{2^{-2n-1}}{\pi^{2n+2}}n!\, L_1(n+1) = -\dfrac{2^{-2n-1}}{\pi^{2n+2}}n!\, \displaystyle\sum_{m=1}^{\infty}\dfrac{A_m-A_{-m}}{m^{n+1}} & \text{ if } n \in \IZ_{\ge 0} \text{ even}\\
        & \\
          \dfrac{2^{-2n-1}}{\pi^{2n+2}}n!\, L_0(n+1) = \dfrac{2^{-2n-1}}{\pi^{2n+2}}n!\, \displaystyle\sum_{m=1}^{\infty}\dfrac{A_m+A_{-m}}{m^{n+1}}  & \text{ if } n \in \IZ_{> 0} \text{ odd}
        \end{cases}\\
        &=\frac{2^{-2n-1}}{\pi^{2n+2}}n! \left(\sum_{m=1}^{\infty} \frac{A_m}{(-m)^{n+1}}+\sum_{m=1}^{\infty} \frac{A_{-m}}{m^{n+1}}\right) \\
        &= \frac{2^{-2n-1}}{\pi^{2n+2}}n! \sum_{m \in \IZ_{\ne 0}} \frac{A_m}{(-m)^{n+1}} \, .
    \end{aligned}
    \ee
    Analogously, putting together Eqs.~\eqref{eq: L0-L1-Lminus-def},~\eqref{eq: L1-n}, and~\eqref{eq: L0-n}, we find that 
    \be \label{eq: Lminus-n}
    \begin{aligned}
        \dfrac{1}{n!}L_{-}(-n)&=\begin{cases} -\dfrac{2^{-2n-1}}{\pi^{2n+2}}n!\, L_1(n+1) = -\dfrac{2^{-2n-1}}{\pi^{2n+2}}n!\, \displaystyle\sum_{m=1}^{\infty}\dfrac{A_m-A_{-m}}{m^{n+1}} & \text{ if } n \in \IZ_{\ge 0} \text{ even}\\
        & \\
          -\dfrac{2^{-2n-1}}{\pi^{2n+2}}n!\, L_0(n+1) = -\dfrac{2^{-2n-1}}{\pi^{2n+2}}n!\, \displaystyle\sum_{m=1}^{\infty}\dfrac{A_m+A_{-m}}{m^{n+1}}  & \text{ if } n \in \IZ_{> 0} \text{ odd}
        \end{cases}\\
        &=-\frac{2^{-2n-1}}{\pi^{2n+2}}n! \left(\sum_{m=1}^{\infty} \frac{A_m}{m^{n+1}}+\sum_{m=1}^{\infty} \frac{A_{-m}}{(-m)^{n+1}}\right) \\
        &=- \frac{2^{-2n-1}}{\pi^{2n+2}}n! \sum_{m \in \IZ_{\ne 0}} \frac{A_m}{m^{n+1}} \, .
    \end{aligned}
    \ee
    To sum up, substituting Eq.~\eqref{eq: L-n} into Eq.~\eqref{eq: ftilde-L-proof} yields 
    \be \label{eq: LZ-ftildeP}
    \begin{aligned}
    \tilde{f}_{+}(y)&=2\sum_{n=0}^\infty n!\sum_{m\in\IZ_{\ne 0}}\frac{A_m}{(-4\pi^2m)^{n+1}}(2 \pi \ri)^n y^{n}\, \\
    &=\sum_{n=0}^\infty \left( \frac{n!}{\pi \ri}\sum_{m\in\IZ_{\ne 0}}\frac{A_m}{(2 \pi \ri m)^{n+1}}\right) y^{n}\,,
    \end{aligned}
    \ee
    which is a Gevrey-1 asymptotic series. Using Proposition~\ref{prop:stokes}, we conclude that the Borel transform of $y\tilde{f}_{+}(y)$ has a modular resurgent structure with simple poles at $\rho_m=2\pi \ri m$, $m\in\IZ_{\ne 0}$, and Stokes constants $2 A_m$. 
    Analogously, substituting Eq.~\eqref{eq: Lminus-n} into Eq.~\eqref{eq: ftilde-Lminus-proof} gives 
    \be \label{eq: LZ-ftildeM}
    \begin{aligned}
    \tilde{f}_{-}(y)&=-2\sum_{n=0}^\infty n!\sum_{m\in\IZ_{\ne 0}}\frac{A_m}{(4\pi^2m)^{n+1}}(2 \pi \ri)^n y^{n}\, \\
    &=-\sum_{n=0}^\infty \left( \frac{n!}{\pi \ri}\sum_{m\in\IZ_{\ne 0}}\frac{A_{-m}}{(2 \pi \ri m)^{n+1}}\right) y^{n}\,,
    \end{aligned}
    \ee
    which is a Gevrey-1 asymptotic series. Once more, using Proposition~\ref{prop:stokes}, we conclude that the Borel transform of $y\tilde{f}_{-}(y)$ has a modular resurgent structure with simple poles at $\rho_m=2\pi \ri m$, $m\in\IZ_{\ne 0}$, and Stokes constants $-2 A_{-m}$. 
\end{proof}

Notice that the MRSs of Proposition~\ref{prop:lewis-Zagier-qms} satisfy the simplified diagram in Eq.~\eqref{diag:simple resur}.
Additionally, the generating function defined in Eq.~\eqref{eq: LZ-f} satisfies Conjecture~\ref{conj:quantum_modular2}.
\begin{theorem}\label{thm:lewis-Zagier_QM}
    Within the assumptions of Proposition~\ref{prop:lewis-Zagier-qms}, the generating function $f:\IC\setminus\IR\to\IC$, defined in Eq.~\eqref{eq: LZ-f}, restricted to either the upper or the lower half of the complex plane is a weight-$1$ holomorphic quantum modular form for the group $\mathsf{SL}_2(\IZ)$. 
\end{theorem}
\begin{proof}
    Recall from part $(d)$ of Theorem~$1$ in~\cite[Chapter~1]{lewis-Zagier--period} that there exists an analytic function $\psi\colon\IC'\to\IC$ such that\footnote{The function $\psi$ in Eq.~\eqref{eq: psi-LZ} is the so-called \emph{period} of the function $f$ in Eq.~\eqref{eq: LZ-f} and satisfies a three-terms functional equation. See~\cite{lewis-Zagier--period} for details.} 
    \begin{equation}\label{eq: psi-LZ}
      c\,\psi(y)=f(y)-\frac{1}{y} f\left(-\frac{1}{y}\right) \, ,
    \end{equation}
for some constant $c \in \IC_{\ne 0}$. Moreover, recall that the generators of the modular group $\mathsf{SL}_2(\IZ)$ are
\be \label{eq:gen_TS}
    T=\begin{pmatrix}
        1 & 1\\
        0 & 1
    \end{pmatrix}\,, \quad  S=\begin{pmatrix}
    0 & -1 \\
    1 & 0
\end{pmatrix}\,.
\ee
Hence, if we restrict the generating function to the upper half-plane, \emph{i.e.}, we take $f\colon\IH\to\IC$, the cocycle for $T$ is trivial due to the periodicity property $f(y+1)=f(y)$, while the cocycle for $S$ is given by
\be
h_S[f](y)=-c\,\psi(y) \, .
\ee
The same arguments apply to the restriction of the generating function to $\IH_-$.
\end{proof}

With an additional assumption on the parity of the coefficients $A_m$, $m \in \IZ_{\ne 0}$, appearing in the definition of the $L$-functions $L_\epsilon$, $\epsilon=0,1$, in Eq.~\eqref{eq: lseries-LZ}, we prove the effectiveness of the median resummation for the generating function $f$ in Eq.~\eqref{eq: LZ-f}, which therefore satisfies Conjecture~\ref{conj:quantum_modular1}. In particular, assuming $A_m=A_{-m}$, we have that $f(-y)=-f(y)$ for $y \in \IC\setminus\IR$, so that $f(y)$ only needs to be specified in the upper half of the complex $y$-plane.
\begin{theorem}\label{thm:lewis-Zagier_median}
    Within the assumptions of Proposition~\ref{prop:lewis-Zagier-qms}, if $A_m=A_{-m}$ for every $m \in \IZ_{\ne 0}$, the generating function $f\colon\IH \to \IC$ defined in Eq.~\eqref{eq: LZ-f} is recovered from its asymptotic expansion $\tilde{f}_+(y)$ as $y\to 0$ with $\Im(y)>0$ through the median resummation. Precisely,
    \be
        f(y)=\mathcal{S}^{\mathrm{med}}_{\frac{\pi}{2}}\big[\tilde{f}_+\big](y) \, , \quad y \in \IH \, .
    \ee
\end{theorem}
\begin{proof} 
Due to the parity of the coefficients, $L_1$ is identically zero, while 
\be
L_0(s) = 2 \sum_{m=1}^{\infty} \frac{A_m}{m^s} \, ,
\ee
so that the two $L$-functions in Eqs.~\eqref{eq: L0-L1-L-def} and~\eqref{eq: L0-L1-Lminus-def} reduce to 
\be
L_\pm(s)=\frac{L_0(s)}{2} \, .
\ee
As a consequence of Proposition~\ref{prop:lewis-Zagier-qms}, and particularly Eqs.~\eqref{eq: LZ-ftildeP} and~\eqref{eq: LZ-ftildeM}, the formal power series $\tilde{f}_{\pm}$ can be written as
\begin{equation}
    \tilde{f}_{\pm}(y)=\pm\frac{2}{\pi \ri}\sum_{n=0}^\infty (2n+1)!\sum_{m>0}\frac{A_m}{\rho_m^{2n+2}}y^{2n+1} \, ,
\end{equation}
where $\rho_m = 2 \pi \ri m$, $m \in \IZ_{>0}$, as before.
Recall that the median resummation at angle $\pi/2$ can be expressed in terms of the Borel--Laplace sums at angles $0$ and $\pi$ using Eq.~\eqref{eq: median2}. We have that
\be \label{eq: median-up}
\begin{aligned}
    \mathcal{S}^{\mathrm{med}}_{\frac{\pi}{2}}\big[y\tilde{f}_+\big](y) &=\begin{cases}
        s_0\big[y\tilde{f}_+\big](y) + \frac{1}{2}\mathrm{disc}_{\frac{\pi}{2}}\big[y\tilde{f}_+\big](y)\,, & \Re(y)>0 \, ,\\
        \\
        s_\pi\big[y\tilde{f}_+\big](y) - \frac{1}{2}\mathrm{disc}_{\frac{\pi}{2}}\big[y\tilde{f}_+\big](y)\,, & \Re(y)<0\,,
    \end{cases}
\end{aligned}
\ee
for $\Im(y)>0$. 
To compute the Borel--Laplace sums of $\tilde{f}_+$, let us first evaluate its Borel transform as in the proof of Proposition~\ref{prop:stokes}, that is,  
\be
\begin{aligned}
    \borel\big[y\tilde{f}_+\big](\zeta) =\frac{2}{\pi \ri} \sum_{n=0}^\infty \sum_{m>0} \frac{A_m}{\rho_m^{2n+2}} \zeta^{2n+1} =-\frac{1}{\pi \ri}\sum_{m\in\IZ_{\neq 0}} \frac{A_m}{\zeta-\rho_m} \, ,
\end{aligned}
\ee
where we have exchanged the order of summation, resummed the geometric series over the index $n$, and applied the identity \begin{equation} \label{eq: identity-frac}
    \frac{1}{(x-y)(x+y)}= \frac{1}{2x(x-y)}+\frac{1}{2x(x+y)}
\end{equation}
in the second step.
Then, the Borel--Laplace sums at angles $0$ and $\pi$ are
\begin{subequations}
    \begin{align}
    s_0\big[y\tilde{f}_+\big](y)&=-\frac{1}{\pi\ri}\int_0^{\infty } \re^{-\zeta/y} \sum_{m\in\IZ_{\neq 0}}\frac{A_m}{\zeta-\rho_m}\, d\zeta=- 2\sum_{m\in\IZ_{\neq 0}} A_m \e\left(-\frac{m}{y}\right) \, , \quad  & y>0 \, , \label{eq: BL-PMzero}\\ 
    s_{\pi}\big[y\tilde{f}_+\big](y)&=-\frac{1}{\pi\ri}\int_0^{-\infty} \re^{-\zeta/y} \sum_{m\in\IZ_{\neq 0}}\frac{A_m}{\zeta-\rho_m}\,  d\zeta=-2\sum_{m\in\IZ_{\neq 0}} A_m \e\left(-\frac{m}{y}\right)\, , \quad & y<0 \,, \label{eq: BL-PMpi}
    \end{align}
\end{subequations}
where we have introduced the analytic function $\e\colon\IC\setminus \ri\IR_{\geq 0}\to\IC$ defined by\footnote{The notation in Eq.~\eqref{eq: e1-def} was suggested by M. Kontsevich as the function $\e$ plays a similar role in various examples (\emph{e.g.}, the case of $\sigma, \sigma^*$ in~\cite[from min.~21.55]{Fantini-talk-IHES} and the spectral trace of local $\IP^2$ in~\cite{FR1phys}).}  
\begin{equation} \label{eq: e1-def}
    \e(z):=\frac{1}{2\pi \ri}\int_0^{\infty} \re^{-2\pi t} \frac{dt}{t+\ri z} =\frac{1}{2\pi \ri} \re^{2\pi \ri z} \, \Gamma(0,2\pi \ri z) \, ,
\end{equation}
where $\Gamma(s, 2\pi \ri z)$ denotes the upper incomplete gamma function. 
Note that the Borel--Laplace sums above can be analytically continued to $\Re(y)>0$ and $\Re(y)<0$, respectively, as a consequence of the properties of the function $\e$.
Besides, by the standard residue argument, the discontinuity is
\be
    \mathrm{disc}_{\frac{\pi}{2}}\big[y\tilde{f}_+\big](y)=2\sum_{m>0}A_m \re^{-2\pi\ri m/y}=2f\left(-\frac{1}{y}\right) \,, \quad \Im(y)>0\,. \label{eq: discfP}
\ee

We can then apply the results of Section~3, Chapter~II of~\cite{lewis-Zagier--period}. 
In particular, adopting the notation of~\cite[Eqs.~(2.11) and~(2.13)]{lewis-Zagier--period}, we define the function $\psi_1\colon\IC\setminus\ri \IR_{\ge 0} \to \IC$ as
\begin{equation} \label{eq: psi1-LZ}
    \psi_1(y):=2\pi \ri\sum_{m\in\IZ_{\neq 0}} A_m \e(my) \, ,
\end{equation}
which satisfies the properties~\cite[Eq.~(2.10)]{lewis-Zagier--period}
\be \label{eq: psi1-id}
 y\psi_1(y)=\psi_1\left(\frac{1}{y}\right)\, , \quad \psi_1(-y)=\psi_1(y) \,,
\ee 
where the first formula holds for $\Re(y)>0$. Then, we have that~\cite[Eq.~(2.20)]{lewis-Zagier--period}\footnote{The same formula in Eq.~\eqref{eq: f_LZ} holds for $y \in \IH_-$ with a change of sign in the constant $c_*$.}
\be\label{eq: f_LZ}
c_*\,f(y)=\psi_1(y)+\frac{1}{y}\psi_1\left(-\frac{1}{y}\right)\, ,
\ee
for some constant $c_* \in \IC_{\ne0}$.
Let us set the constant $c_*=-2\pi\ri$. Using Eqs.~\eqref{eq: f_LZ} and~\eqref{eq: psi1-id}, we deduce that
\be \label{eq: fftilde-up}
\begin{aligned}
    y f(y)&=\begin{cases}
     \displaystyle -\frac{1}{\pi\ri}\psi_1\left(-\frac{1}{y}\right) + f\left(-\frac{1}{y}\right)  \, , \quad & \Re(y)>0 \, , \\
     & \\
    \displaystyle -\frac{1}{\pi\ri}\psi_1\left(-\frac{1}{y}\right) -f\left(-\frac{1}{y}\right) \, , \quad & \Re(y)<0 \, . 
    \end{cases}
\end{aligned}
\ee
It follows from Eqs.~\eqref{eq: psi1-LZ},~\eqref{eq: BL-PMzero}, and~\eqref{eq: BL-PMpi} that we can identify $\psi_1(-1/y)$ with the Borel--Laplace sum of $y \tilde{f}_+(y)$ along the positive real axis for $\Re(y)>0$, that is,
\be \label{eq: BL-Pzero}
-\frac{1}{\pi\ri}\psi_1\left(-\frac{1}{y}\right)=s_0\big[y\tilde{f}_{+}\big](y)\,, \quad \Re(y)>0 \, ,
\ee
and with its Borel--Laplace sum along the negative real axis for $\Re(y)<0$, that is, 
\be \label{eq: BL-Ppi}
-\frac{1}{\pi\ri}\psi_1\left(-\frac{1}{y}\right)=s_\pi\big[y\tilde{f}_{+}\big](y) \, , \quad \Re(y)<0 \, .
\ee
Substituting the discontinuity in Eq.~\eqref{eq: discfP} and the Borel--Laplace sums in Eqs.~\eqref{eq: BL-Pzero} and~\eqref{eq: BL-Ppi} into the RHS of the formula for the median resummation in Eq.~\eqref{eq: median-up}, we obtain\footnote{We recall that the Borel--Laplace sum of an asymptotic series $\phi \in \IC [\![y]\!]$ satisfies the property $s_\theta[y\phi\,]=y\,s_\theta[\phi]$.} Eq.~\eqref{eq: fftilde-up} and conclude. 
\end{proof}

A result analogous to Theorem~\ref{thm:lewis-Zagier_median} applies for an odd sequence of complex numbers $A_m$, $m \in \IZ_{\ne 0}$, \emph{i.e.}, such that $A_m=-A_{-m}$. In this case, the generating function in Eq.~\eqref{eq: LZ-f} satisfies $f(-y)=f(y)$ for $y \in \IC\setminus\IR$, and $L_1(s)=2 \sum_{m=1}^{\infty} \tfrac{A_m}{m^s}$ is the only non-trivial $L$-function. Moreover, we expect a similar result for Maass cusps forms that are invariant under a discrete subgroup of $\mathsf{PSL}_2(\IZ)$.

\begin{rmk}
The example of even Maass cusp forms discussed above displays an interesting feature. Indeed, the cocycle $\psi(y)$ in Eq.~\eqref{eq: psi-LZ} is proportional to the function $\psi_1(y)$ in Eq.~\eqref{eq: psi1-LZ} by means of Eq.~\eqref{eq: f_LZ}. Consequently, it follows from Theorems~\ref{thm:lewis-Zagier_QM} and~\ref{thm:lewis-Zagier_median} that the Borel--Laplace sums $s_0\big[y\tilde{f}_\pm\big]$ are proportional to $\psi(-1/y)$. We expect this property to hold more generally when both Conjectures~\ref{conj:quantum_modular1} and~\ref{conj:quantum_modular2} (or, equivalently, Conjectures~\ref{conj:quantum_modular1-Q} and~\ref{conj:quantum_modular2-Q}) are verified---potentially, under some additional assumptions on the $q$-series. In other words, we expect the cocycles of (a well-defined class of) $q$-series whose asymptotic expansions have a modular resurgent structure to occur as Borel--Laplace sums. We leave further investigation to future work.
\end{rmk}

\subsection{From cusp forms}\label{sec:modular}
In this section, we consider the $L$-function of a general cusp form. Breaking the modular invariance of the cusp form by adding to it a function written in terms of the exponential integral $\mathbf{e}_1$ in Eq.~\eqref{eq: e1-def}, we build a holomorphic function whose asymptotic series is modular resurgent. In particular, its Stokes constants are the Fourier coefficients of the original cusp form. Additionally, we show that the said holomorphic function, although not expressible in the form of a $q$-series, satisfies the theses of Conjectures~\ref{conj:quantum_modular1} and~\ref{conj:quantum_modular2}.
\begin{theorem}\label{thm:modular}
Let $g\colon\IH\to\IC$ be a cusp form of weight $\omega\in\frac{1}{2}\IZ$ for $\mathsf{SL}_2(\IZ)$ whose Fourier coefficients are the complex numbers $A_m$, $m\in\IZ_{>0}$. Let $f\colon\IH\to\IC$ be the function
\be\label{eq:f_qm}
f(y):=\frac{1}{2}g\left(-\frac{1}{y}\right)+\sum_{m>0} A_m \mathbf{e}_1\left(-\frac{m}{y}\right) \, ,
\ee
where $\mathbf{e}_1$ is defined in Eq.~\eqref{eq: e1-def}. Then, the Gevrey-1 asymptotic series $\tilde{f} \in\IC\llbracket y\rrbracket$ obtained expanding $f(y)$ for $y \to 0$ with $\Im(y)>0$ has a \emph{modular resurgent structure} with simple poles at $\zeta_m=2\pi\ri m$, $m\in\IZ_{>0}$, and corresponding Stokes constants $-A_m$. 
In addition, 
\begin{itemize}
  \item $\CS^{\mathrm{med}}_{\frac{\pi}{2}}[\tilde{f}](y)=f(y)$ for every $y\in\IH$;
  \item $f(y)$ is a holomorphic quantum modular form of weight $-\omega$ for $\mathsf{SL}_2(\IZ)$.
\end{itemize}
\end{theorem}
\begin{proof}
Recall that the asymptotic expansion of the exponential function $\mathbf{e}_1(z)$ in Eq.~\eqref{eq: e1-def} in the limit $z \to \infty$ is~\cite[\href{https://dlmf.nist.gov/8.11.E2}{(8.11.E2)}]{NIST:DLMF}
\be
\mathbf{e}_1(z) = \frac{1}{2 \pi \ri} \sum_{n=0}^\infty \frac{ (-1)^n n!}{(2 \pi \ri z)^{n+1}} \, .
\ee
It follows that the formal power series $\tilde{f}$ reads
\be
\tilde{f}(y)=-\frac{1}{2\pi\ri}\sum_{m>0}A_m\sum_{n=0}^\infty\frac{n!}{(2\pi\ri m)^{n+1}}y^{n+1} =-\frac{1}{2\pi\ri}\sum_{n=1}^\infty \sum_{m>0}\frac{A_m}{(2\pi\ri m)^{n}} \Gamma(n) y^{n}
\, .
\ee
Hence, its perturbative coefficients are of the form in Eq.~\eqref{eq: cn-Am} with $\zeta_m=2\pi\ri m$. By Proposition~\ref{prop:stokes}, $\tilde{f}$ has a simple resurgent structure with a tower of simple poles located at $\zeta_m$ for $m\in\IZ_{>0}$ and corresponding Stokes constants $-A_m$. Moreover, since the $L$-series 
\be
L_g(s):=\sum_{m>0}\frac{A_m}{m^s}
\ee 
admits a meromorphic continuation to the whole complex $s$-plane due to a well-known result by Hecke~\cite[Section~5]{Lang1987}, $\tilde{f}$ is modular resurgent.  

Let us now compute the median resummation applying Eq.~\eqref{eq: median2}. Let $\Im(y)>0$ and fix a small angle $\theta>0$. We have that
\be
\ba
\CS^{\mathrm{med}}_{\frac{\pi}{2}}[\tilde{f}](y)&=-\frac{1}{2\pi\ri}\int_0^{\ri \re^{\ri\theta}\infty} \re^{-\zeta/ y} \sum_{n=1}^\infty \sum_{m>0}\frac{A_m}{(2\pi\ri m)^{n}}\zeta^{n-1} d\zeta+\frac{1}{2}\sum_{m>0}A_m \re^{-2\pi\ri m/ y} \notag\\
&=\frac{1}{2\pi\ri}\sum_{m>0}A_m \, \int_0^{\ri \re^{\ri\theta}\infty} \re^{-\zeta/ y} \frac{d\zeta}{\zeta-2\pi\ri m} +\frac{1}{2} g\left(-\frac{1}{y}\right) \notag\\
&=\frac{1}{2\pi\ri}\sum_{m>0}A_m\, \int_0^{\infty} \re^{-2\pi t} \frac{dt}{t-\frac{\ri m}{y}}+\frac{1}{2} g\left(-\frac{1}{y}\right) \, , 
\ea
\ee
where we used the absolute convergence of the integrand in the second step. The desired result then follows from the definitions in Eqs.~\eqref{eq: e1-def} and~\eqref{eq:f_qm}.

Finally, let us verify the quantum modularity property. 
For simplicity, we choose a pair of generators of $\mathsf{SL}_2(\IZ)$ different from before, namely, $S$ and $\gamma:=TS$, where $T$ and $S$ are the standard generators in Eq.~\eqref{eq:gen_TS}. Then, the corresponding cocycles are
\begin{subequations}
\begin{align}
h_S[f](y)&= y^{\omega}\sum_{m>0} A_m \mathbf{e}_1(my)-\sum_{m>0} A_m \mathbf{e}_1\left(-\frac{m}{y}\right)\,, \\
h_\gamma[f](y)&= y^\omega \sum_{m>0} A_m \mathbf{e}_1\left(-\frac{my}{y-1}\right)-\sum_{m>0} A_m \mathbf{e}_1\left(-\frac{m}{y}\right) \, ,
\end{align}
\end{subequations}
respectively, where we implicitly used the modularity of the cusp form $g$. 
\end{proof}

\section{Conclusions}\label{sec:conclusion}
In this paper, we identified and studied a special class of resurgent Gevrey-1 asymptotic series that appear in various contexts, particularly in topological string theory on toric CY threefolds and in the theory of Maass cusp forms.

Motivated by the recent mathematical and physical interest, we introduced the notion of a \emph{modular resurgent series}, whose Borel plane displays an infinite tower of singularities, the secondary resurgent series are trivial, and the Stokes constants are coefficients of an $L$-function. 
We proposed a new perspective on these \emph{modular resurgent structures} by taking the point of view of the Stokes constants and their generating and Dirichlet series.
Generalizing the strong-weak resurgent symmetry of local $\IP^2$~\cite{Rella22, FR1phys}, we presented the broader paradigm of \emph{modular resurgence}. Certain canonical pairs of MRSs obtained by asymptotically expanding a pair of $q$-series are in fact connected through a global exact symmetry. Here, each $q$-series equals the discontinuity of the asymptotic expansion of the other, while the two $L$-functions whose coefficients are the Stokes constants of the MRSs satisfy a combined functional equation. 
Supported by numerous examples, we conjectured that the median resummation of MRSs that come from the asymptotic expansion of a $q$-series is effective and produces holomorphic quantum modular forms.

As a concrete benchmark, we illustrated the workings of the modular resurgence paradigm via explicit computations in the example of the weak and strong coupling asymptotic expansions of the spectral trace of local $\IP^2$. 
Finally, we proved that a large class of MRSs originating from the theory of Maass cusp forms obeys our conjectural statements. More general cusp forms are also shown to give rise to MRSs, although they satisfy a slightly altered version of the paradigm and conjectures.

Let us highlight three main research directions that arise from the investigation performed in this paper and are to be pursued in follow-up work. 

Firstly, we would like to broaden the supporting basis for our paradigm of modular resurgence. A bigger and more varied pool of examples might help us deepen our understanding of the critical role played by the $L$-functions constructed from the Stokes constants and their functional equation. At the same time, a complete characterization of these ingredients might steer us towards proving Conjectures~\ref{conj:quantum_modular1} and~\ref{conj:quantum_modular2} (and Conjectures~\ref{conj:quantum_modular1-Q} and~\ref{conj:quantum_modular2-Q}), clarifying the relation between them, and their link to the modular resurgence paradigm. 

Secondly, our definition of modular resurgent structure easily lends itself to generalization in different directions. Instructed by topological string theory and complex CS theory, we would like to enlarge the domain of application of modular resurgence to incorporate more general patterns of singularities in the Borel plane, \emph{e.g.}, multiple towers or rays, as well as to make contact with vector-valued quantum modular forms. The paradigm and conjectures would then need to be modified accordingly. 

Finally, our statements on the median resummation could be placed in a more general context. Indeed, as the first author started addressing in~\cite{borel_reg}, when divergent formal power series arise from the asymptotic expansion of well-defined analytic functions, it is natural to ask which summability method allows us to reconstruct the original functions effectively. It would be interesting to clarify the effectiveness of summability methods in a broader sense, following the path opened by our Conjectures~\ref{conj:quantum_modular1} and~\ref{conj:quantum_modular1-Q} and Conjecture~1.1 in~\cite{costin-garoufalidis}.

\section*{Acknowledgements}
We thank 
Alba Grassi, 
Maxim Kontsevich, 
Marcos Mari\~no, 
Campbell Wheeler,
and 
Don Zagier 
for many useful discussions.
This work has been supported by the ERC-SyG project ``Recursive and Exact New Quantum Theory'' (ReNewQuantum), which received funding from the European Research Council (ERC) within the European Union's Horizon 2020 research and innovation program under Grant No. 810573, and by the Swiss National Centre of Competence in Research SwissMAP (NCCR 51NF40-141869 The Mathematics of Physics).

\addcontentsline{toc}{section}{References}
\bibliographystyle{JHEP}
\linespread{0.4}
\bibliography{localP2-biblio}

\end{document}